\documentclass[a4paper,11pt]{amsart}
\pdfoutput=1   %%for arXiv uploads
\usepackage{amsmath, amssymb, amsthm, amsbsy, bbm, amsfonts, color, verbatim, array, floatflt, yfonts}
\usepackage[margin=1.1in]{geometry}
\usepackage[all]{xy}
\usepackage[english]{babel}
\usepackage[shortlabels]{enumitem}

\usepackage[utf8]{inputenc}
\usepackage{indentfirst}
\usepackage{csquotes}
\usepackage{upgreek,color}

\usepackage[toc,page]{appendix}
\usepackage{graphicx}
\usepackage{overpic}
\usepackage{mathtools,array}
\usepackage[all]{xy}

\usepackage{wrapfig}
\usepackage{subfig}

\usepackage{tikz}
\usetikzlibrary{cd}
\usetikzlibrary{fit}
\usetikzlibrary{knots}
\usetikzlibrary{positioning}
\usetikzlibrary{arrows.meta}
\usetikzlibrary{patterns}
\usetikzlibrary{backgrounds}
\usetikzlibrary{shapes,shapes.geometric,shapes.misc}
\pgfdeclarelayer{edgelayer}
\pgfdeclarelayer{nodelayer}
\pgfsetlayers{background,edgelayer,nodelayer,main}
\tikzstyle{none}=[inner sep=0mm]

\usepackage{fancyhdr,floatpag}
\usepackage{hyperref}
\usepackage[capitalise]{cleveref}
\usepackage{blkarray}
\usepackage{multirow}
\usepackage{dutchcal} %calligraphic letters in math mode
\hypersetup{linkbordercolor={0 0.60 0.75}, linkcolor=MidnightBlue}
\hypersetup{citebordercolor={0 0.60 0.75}, citecolor=MidnightBlue}

\definecolor{amethyst}{rgb}{0.6, 0.4, 0.8}

\usepackage[style=alphabetic,sorting=nyt,maxbibnames=7,maxcitenames=5,block=space,backend=bibtex]{biblatex}
\renewbibmacro{in:}{}

\numberwithin{equation}{section}
\theoremstyle{plain}

\newtheorem{thm}{Theorem}[section]
\newtheorem{conj}[thm]{Conjecture}
\newtheorem{cor}[thm]{Corollary}
\newtheorem{lem}[thm]{Lemma}

\newtheorem{prop}[thm]{Proposition}

\newtheorem{qst}[thm]{Question}

\theoremstyle{definition}
\newtheorem{dfn}[thm]{Definition}
\newtheorem{rmkx}[thm]{Remark}
\newenvironment{remark}
  {\pushQED{\qed}\rmkx}
  {\popQED\endrmkx}
\newtheorem{examplex}[thm]{Example}
\newenvironment{exm}
  {\pushQED{\qed}\examplex}
  {\popQED\endexamplex}

\newcommand{\N}{\mathbb{N}} %naturals
\newcommand{\Z}{\mathbb{Z}} %integers
\newcommand{\Q}{\mathbb{Q}} %rationals
\newcommand{\R}{\mathbb{R}} %reals
\newcommand{\Sph}{\ensuremath{\mathbb{S}}}
\newcommand{\Hyp}{\ensuremath{\mathbb{H}}}

\newcommand{\tA}{\mathtt{A}}
\newcommand{\tB}{\mathtt{B}}

\newcommand{\tD}{\mathtt{D}}
\newcommand{\tE}{\mathtt{E}}

\newcommand{\tH}{\mathtt{H}}
\newcommand{\tI}{\mathtt{I}}
\newcommand{\ord}[1]{m_{{#1}}} %order
\newcommand{\mst}{{m_{s,t}}} %order
\newcommand{\CCG}[1]{\mathfrak{C}{(#1)}} %{\mathtt{CCG}{(#1)}} % complete Coxeter graph
\newcommand{\CCGo}{\mathfrak{C}} % complete Coxeter graph - ohne Klammern
\newcommand{\VC}{\mathcal{G}^{\mathrm{VC}}}  % command for the vertical core

\newcommand{\phee}{\varphi} %nice phi
\newcommand{\quer}[1]{\overline{#1}}
\newcommand{\hut}[1]{\widehat{#1}}
\newcommand{\til}[1]{\widetilde{#1}}
\newcommand{\spans}[1]{\langle {#1} \rangle}
\newcommand{\Menge}[1]{\{ #1 \}}
\newcommand{\leer}{\varnothing} %nice empty set
 
\newcommand{\mc}{\mathcal} %math caligraphic
\newcommand{\mrm}{\mathrm} %math roman
 
\newcommand{\into}{\hookrightarrow}

\DeclareMathOperator{\Aut}{Aut}

\DeclareMathOperator{\Isom}{Isom}

\newcommand{\PSL}{\mathbb{P}\mathrm{SL}}

\title[The Coxeter galaxy]{The galaxy of Coxeter groups}
\author{Yuri Santos Rego and Petra Schwer}
\address{Otto-von-Guericke-Universit\"at Magdeburg, \newline 
Fakult\"at f\"ur Mathematik -- Institut f\"ur Algebra und Geometrie, \newline 
PSF 4120, 39016 Magdeburg, Deutschland}
\email{yuri.santos@ovgu.de}
\email{petra.schwer@ovgu.de}

\begin{document}
\begin{abstract}
In this paper we introduce the galaxy of Coxeter groups (of finite rank) -- an infinite dimensional, locally finite, ranked simplicial complex which captures isomorphisms between finite rank Coxeter systems. In doing so, we would like to suggest a new framework to study the isomorphism problem for Coxeter groups. 
We prove some structural results about this space, provide a full characterization in small ranks and propose many questions. In addition we survey known tools, results and conjectures. 

Along the way we show profinite rigidity of triangle Coxeter groups -- a result which is possibly of independent interest.  
\end{abstract}
\thispagestyle{empty}

\maketitle

\begin{center}
	\emph{We dedicate this paper to the memory of Jacques Tits for his influential work on Coxeter groups. It was him who, as an honorary Bourbaki, coined the terms {Coxeter group and Coxeter diagram}. \footnote{See: S. Roberts, "Donald Coxeter: The man who saved geometry", 2003. \cite{Roberts}} }
\end{center}

\smallskip

%%%%%%%%%%%%%%%%%%%%%%%%%%%%%%%%%%%%%%%%%%%%%%%%%%
% Introduction
%%%%%%%%%%%%%%%%%%%%%%%%%%%%%%%%%%%%%%%%%%%%%%%%%%%

\section{Introduction}
\label{sec:introduction}

The universe of finitely generated groups is pretty wild. Even distinguishing two of its groups algorithmically is a hopeless task since Dehn's isomorphism problem is undecidable by work of Novikov \cite{Novikov}. On the other hand, the region containing Coxeter groups is fairly well explored by group theorists and it is expected that an algorithm should exist that tells two given Coxeter groups apart.  

In this paper we introduce the \emph{Coxeter galaxy}  - a simplicial complex which encodes Coxeter systems of finite rank and non-type preserving isomorphisms between the underlying Coxeter groups. This space provides a framework to study the isomorphism problem in the class of Coxeter groups as well as  related more refined questions. 

To be a little more precise, the Coxeter galaxy is the flag simplicial complex $\mc{G}$ whose vertices are (graph-isomorphism classes of) Coxeter systems of finite rank, two of which are connected if their underlying Coxeter groups are isomorphic as abstract groups. Solving the isomorphism problem amounts to algorithmically being able to determine the connected components of the galaxy.  
This approach may a priori seem to be just packaging. However, it makes the isomorphism problem more tractable, yields some obvious refinements and  partial versions of the isomorphism problem and allows to naturally formulate related questions. 

\subsection{Structure of the galaxy}
The Coxeter galaxy is  automatically organized into horizontal layers  $\mc{G}_{k}$ containing all Coxeter systems of rank $k$; see \cref{def:Layer}. Its connected components typically span over several layers as distinct Coxeter systems for a given group may have different ranks. 
The subcomplex spanned by Coxeter systems of rank at most $k$ is denoted by $\mc{G}_{\leq k}$. A solution to the isomorphism problem for (finitely generated) Coxeter groups is then equivalent to a solution to all  `height-$k$ restricted isomorphism problems', i.e., for all $\mc{G}_{\leq k}$ for every $k \in \N$. Clearly, a full description of connected components of the galaxy would yield a complete classification of isomorphic Coxeter groups admitting nonisomorphic Coxeter systems. 

In this paper we address multiple questions about the general structure of the Coxeter galaxy. As a consequence of results in the literature, we show that the Coxeter galaxy is locally finite, with finite dimensional connected components; see \cref{subsec:components} and \cref{thm:FinitenessConnComps}. We also identify some useful subspaces, such as arbitrarily large `clusters of (irreducible) little stars' (\cref{lem:starlet}) or the \emph{vertical core} of the galaxy (\cref{def:Core}). 
The latter is a subcomplex obtained by deleting certain vertical edges of the galaxy. We prove that it has the same connected components as $\mc{G}$ (\cref{prop:verticalCore}). Both the galaxy and the vertical core are disconnected and infinite dimensional (\cref{cor:GalaxyandVCareinfinitedimensional}). It would be interesting to know whether the vertical core is a deformation retract of the galaxy. 
We summarize our main structural findings in the following theorem. 

\begin{thm}[{Structure of the Coxeter galaxy $\mc{G}$}]
\label{thm:main-structure}
 {\ }
     \begin{enumerate}
         \item The galaxy $\mc{G}$ is a locally finite, infinite dimensional simplicial complex with finite connected components, all of which are simplices. (See \cref{thm:FinitenessConnComps} and \cref{cor:GalaxyandVCareinfinitedimensional}.)
         \item Solving the isomorphism problem reduces to algorithmically computing (a spine of) the vertical core. (See \cref{prop:verticalCore}.) 
         \item The subcomplex $\mc{G}_{\leq 3}$  is a 1-dimensional complex and equal to its vertical core. Furthermore, the isomorphism problem is decidable for groups in this subcomplex. (See Corollaries~\ref{cor:Rank3is1D} and~\ref{cor:SolutionIsoProbRk3}.) 
     \end{enumerate}
\end{thm}

\noindent Many questions, such as \cref{qst:Growth of dimension} about the growth rate of the dimensions of each layer, are posed throughout the text. See, for example, Questions~\ref{qst:refinedGalactic}, \ref{qst:refinedSpine}, \ref{qst:diameter} and \ref{qst:reachability}. 

%%%%%%%%%%%%%%%%%%%%%%%%%%%%%%%%%%%%%%%%
\subsection{Layers and profinite rigidity}   
The last item in \cref{thm:main-structure} uses profinite techniques; cf. \cref{subsec:rank3} for basic notions and references. We paint a complete picture of the first few layers $\mc{G}_{\leq 3}$ of the galaxy in \cref{subsec:rank3} by showing that Coxeter groups in $\mc{G}_{\leq 3}$ are {profinitely rigid} within that family. 
This extends a result of Bridson--Conder--Reid \cite[Section~8]{BridsonConderReid}. 
We summarize Theorems~\ref{thm:ProfRigHypRank3} and \ref{thm:analogueBCR} as \cref{thm:main-profinite} below. 

\begin{thm}[Profinite rigidity in rank $\leq 3$]
\label{thm:main-profinite}
{\ }

\noindent    Coxeter groups with diagrams in $\mc{G}_{\leq 3}$ are profinitely rigid within this class, that is, two such groups are isomorphic if and only if their profinite completions agree. 
    Moreover, isomorphism of triangle groups can be decided by comparing the multisets of edge labels of their defining complete Coxeter graphs as defined in \cref{def:CCG}.
\end{thm}

As a consequence, we obtain a new proof that the `height-$3$ restricted isomorphism problem' is decidable, see item~(iii) of \cref{thm:main-structure}. The proof of \cref{thm:main-profinite} highlights how profinite completions can yield information about the isomorphism type of Coxeter groups. We expect that the idea of using finite quotients (besides finite subgroups) to distinguish Coxeter groups will lead to further developments.

\subsection{Navigating the galaxy} 
The classical approach to the isomorphism problem aims to reduce the problem to a specific subclass of Coxeter groups (e.g. angle- or reflection-compatible Coxeter groups) and to finding explicit moves between such systems. These moves often amount to explicit manipulations of the underlying defining diagram and correspond to edges in the Coxeter galaxy. 
Finding all such moves is an interesting problem on its own which we posed as \cref{qst:reachability}. 
This reachability problem immediately raises two related questions. The first can be viewed as an `edge coloring problem' for the galaxy, where colors correspond to the types of moves that yield a certain edge. 

\begin{qst}[Coloring problem]
Suppose \cref{qst:reachability} is solvable. Can the algorithm also output which kinds of edges appear in a path between vertices of the galaxy? Similarly, can we classify which types of moves (i.e., maps between Coxeter systems) transform a complete Coxeter graph into another without changing the underlying Coxeter group?
\end{qst}

This question has been (implicitly) asked multiple times in the literature, and has been answered positively in some cases. For instance, known results e.g., by Mihalik, Ratcliffe, and Tschantz~\cite{MihalikRatcliffeTschantz,MihalikRatcliffe-Ranks}, 
imply that the so-called blow-ups along pseudo-transpositions, introduced by Howlett and Mühlherr in \cite{HowlettMuehlherr} (see also \cite{BernhardSurvey}) suffice to vertically navigate the galaxy. See \cref{cor:navigatingverticallayers} and \cref{thm:downwardsnavigation} for precise statements. In particular, each vertical edge decomposes as a sequence of vertical edges that roughly correspond to blow-ups between adjacent layers. Hence blow-ups and -downs can be thought of as vertical atomic moves. The corresponding horizontal question is still open. 

\begin{qst}[Atomic horizontal moves]
    Is it possible to determine all existing kinds of horizontal moves? Does there exist a finite list of atomic moves such that every move decomposes as a concatenation of those? 
\end{qst}

The current state of the art revolves around developments related to M\"uhlherr's twist conjecture; cf. \cref{conj:ac-twistConjecture}. See, for example,  \cite{MuehlherrWeidmann,BMMN,BernhardSurvey,CapraceMuehlherr2-Spherical,RatcliffeTschantz,CapracePrzytyckiTwistRigid,WeigelTwistTriangle,NuidaSurvey,HuangPrzytycki}. \cref{sec:History} contains an in-depth discussion of this conjecture along with illustrations of other known horizontal moves. Since the Coxeter galaxy is locally finite, though, we ask the following weaker version of the twist conjecture: is there an upper bound on the total number of possible moves between Coxeter systems?

\subsection{Some comments on methods and proofs} Most of the present work draws from multiple sources. We attempted a thorough overview of the literature in \cref{sec:History}, but also throughout the body of the work while translating some of the known results into our language.

The motivating problem for our work is to paint a complete picture of Coxeter systems that have isomorphic underlying groups. Obviously, some complexity reductions towards such an algorithm might be very useful. There are some known reduction steps in the literature that motivate the study of groups with certain restrictions on their Coxeter generating sets, such as reflection- or angle-compatibility; cf. \cref{subsec:twistConj}. However, since the main reason for introducing our framework is to give a {complete} picture of all Coxeter systems (and how to navigate between them), we intentionally avoid such reductions, or aim at translating some reductions into our language with explicit arguments. As a byproduct, all results proved here either use new methods (as those in \cref{subsec:rank3}) or employ results whose proof techniques are independent of developments around the twist conjecture, e.g. \cref{cor:navigatingverticallayers}. We make no use of automorphisms of Coxeter groups.

\subsection{Organization of the paper} 
Notation is set and some basic results are recalled in \cref{sec:Prelim}. The Coxeter galaxy $\mc{G}$ and its vertical core $\VC$ are defined in \cref{subsec:galaxyDefinition} with reformulations of the isomorphism problem in \cref{subsec:Isoproblem}. \cref{sec:FirstExpedition} is devoted to collecting results on vertical moves in \cref{subsec:moves} and proving structural results about the Coxeter galaxy, such as local finiteness and properties of useful subcomplexes in \cref{subsec:components}. In \cref{subsec:rank3} we give a full description of the galaxy in smaller ranks along with results on profinite rigidity of Coxeter groups. The final \cref{sec:History} is an exposition on the known horizontal moves between Coxeter systems (see \cref{subsec:horizontalmoves}), with focus on the twist conjecture in \cref{subsec:twistConj}. We also discuss some reductions of the isomorphism problem in \cref{subsec:ReductionIsoProblem} and summarize other known results in \cref{subsec:survey}. % It includes a `picture' of the Coxeter galaxy taking into account earlier developments.

%%%%%%%%%%%%%%%%%%%%%%%%%%%%%%%%%%%%%%%%%%%%%%%%%%
% Part 0: Preliminaries
%%%%%%%%%%%%%%%%%%%%%%%%%%%%%%%%%%%%%%%%%%%%%%%%%%%

\section{Preliminaries} 
\label{sec:Prelim}

In order to fix notation and to avoid ambiguity we recall basic definitions and results around Coxeter groups. Please note that all Coxeter groups considered in this work are of finite rank. Feel free to skip and refer back to this section later if you prefer to start the galactic journey right away. 

\begin{dfn}[Coxeter matrix]
\label{def:CoxeterMatrix}
A \emph{Coxeter matrix} $M_S$ for a given finite set $S$ is a symmetric $\vert S \vert \times \vert S \vert$-matrix whose entries $\mst$, with  $s,t \in S$, satisfy the following conditions: all diagonal entries $\ord{s,s}$ equal $1$ and all off-diagonal entries belong to $\N_{\geq 2} \cup \Menge{\infty}$.
\end{dfn}

\begin{dfn}[Coxeter groups and systems] 
\label{def:CoxeterSystem}
A \emph{Coxeter system} consists of a triple $(W,S,M_S)$ where $W$ is a group, $S$ a finite subset of $W$ is a generating set, and $M_S$ is a Coxeter matrix for $S$ for which $W$ admits a presentation of the following form:
\[
 W \cong \spans{s \in S \mid (st)^\mst \text{ for every pair } s,t\in S \text{ with } \mst \neq \infty}. 
 \]
A \emph{Coxeter group} is the underlying group $W$ arising from a Coxeter system $(W,S,M_S)$. The set $S \subseteq W$ is usually called a \emph{Coxeter generating set} for $W$. The \emph{type} (or \emph{Coxeter type}) of a Coxeter system $(W,S,M_S)$ is the pair $(S,M_S)$. We say that two Coxeter systems $(W_1,S_1,M_{S_1})$ and $(W_2,S_2,M_{S_2})$ \emph{have the same type} (or are \emph{isomorphic}) if there is a bijection $f \colon S_1 \to S_2$ such that, for each $s,t\in S_1$, the corresponding entry $m_{f(s),f(t)}$ in $M_{S_2}$ is equal to the original entry $m_{s,t}$ in $M_{S_1}$.
\end{dfn} 

For brevity, one typically omits the Coxeter matrix in the notation and speaks about `the' Coxeter system $(W,S)$. We remark that two systems of the same type have isomorphic underlying Coxeter groups.

\begin{dfn}[Spherical Coxeter system, longest element] 
\label{def:sphericalCoxeterSystem}
A Coxeter system $(W,S)$ is called \emph{spherical} if its underlying Coxeter group $W$ is finite. The \emph{longest element} of $W$ is the unique element $w_0 \in W$ of maximal word length with respect to the Coxeter generating set $S$. (Such a $w_0$ exists if and only if $W$ is spherical~\cite[Lemma~4.6.1]{DavisBook}.)
\end{dfn}

\begin{dfn}[Parabolics] 
\label{def:parabolic}
A subgroup $W_T = \spans{T} \leq W$ generated by $T \subseteq S$ in a Coxeter system $(W,S,M_S)$ is called a \emph{standard parabolic} (or \emph{special} or \emph{visible}) subgroup of $W$. Any $P\leq W$ conjugate to some standard parabolic is called \emph{parabolic}. By \cite[Theorem~4.1.6]{DavisBook}, a parabolic subgroup is (isomorphic to) a Coxeter group with underlying system $(W_T,T,M_T)$, where $M_T$ is the (sub)matrix obtained from $M_S$ by taking only entries in $T\times T$.
\end{dfn} 

As usual in the theory of Coxeter groups, it is helpful to encode Coxeter systems as certain graphs. Two popular such graphs are the following. 

\begin{dfn}[Coxeter--Dynkin diagram and defining graph] 
\label{def:CoxDiag}
The \emph{Coxeter--Dynkin diagram} of a Coxeter system $(W,S,M_S)$ is the undirected edge-labeled graph with vertex set $S$, where {distinct} $s,t\in S$ are joined by an edge if $m_{s,t}$ lies in $\N_{\geq 3} \cup \Menge{\infty}$, in which case the edge is unlabeled if $\mst=3$ and otherwise $m_{s,t}$ is the label of the given edge. \newline 
Similarly, the \emph{defining graph} of $(W,S,M_S)$ is the undirected labeled graph whose vertex set is $S$, but now with (distinct) vertices $s,t \in S$ connected by an edge labeled $m_{s,t}$ in case $m_{s,t} \neq \infty$.
\end{dfn}

\begin{remark}
Coxeter--Dynkin diagrams are in the literature sometimes called Dynkin diagrams, Coxeter diagrams, or C-diagrams. 
The {defining graph} of a Coxeter system is in some places referred to as its presentation diagram or P-diagram. 
Here we stick to the notions defined in \cref{def:CoxDiag}. 
\end{remark}

Since it is useful to consider all orders $\mst$ at once (hence have all of them depicted), we introduce a third type of graph representing Coxeter systems.

\begin{dfn}[Complete Coxeter graph] \label{def:CCG}
Given a Coxeter system $(W,S,M_S)$, its \emph{complete Coxeter graph}, denoted by $\CCG{W,S}$, is the complete simplicial graph with vertices the elements of $S$ and with edges labeled by the orders $\mst \in M_S$ for $s\neq t$. When the Coxeter system is implicit from context, we sometimes also write $\CCG{W}$ or simply $\CCGo$ for the complete Coxeter graph of $(W,S,M_S)$. The (Coxeter) \emph{type} of $\CCG{W,S}$ is the (Coxeter) type of the underlying system $(W,S,M_S)$. (Equivalently, the type of $\CCG{W,S}$ is its \emph{isomorphism type} as an edge-labeled graph.)
\end{dfn}

For easier visualization, we draw in a complete Coxeter graph its $\infty$-labeled edges dashed, and its $2$-labeled edges thick and grey. We have illustrated the three types of graphs with an example in Figure~\ref{fig:alldiagrams}. 

\begin{figure}[h] \label{fig:alldiagrams}
\centering
\begin{overpic}[scale=.35]{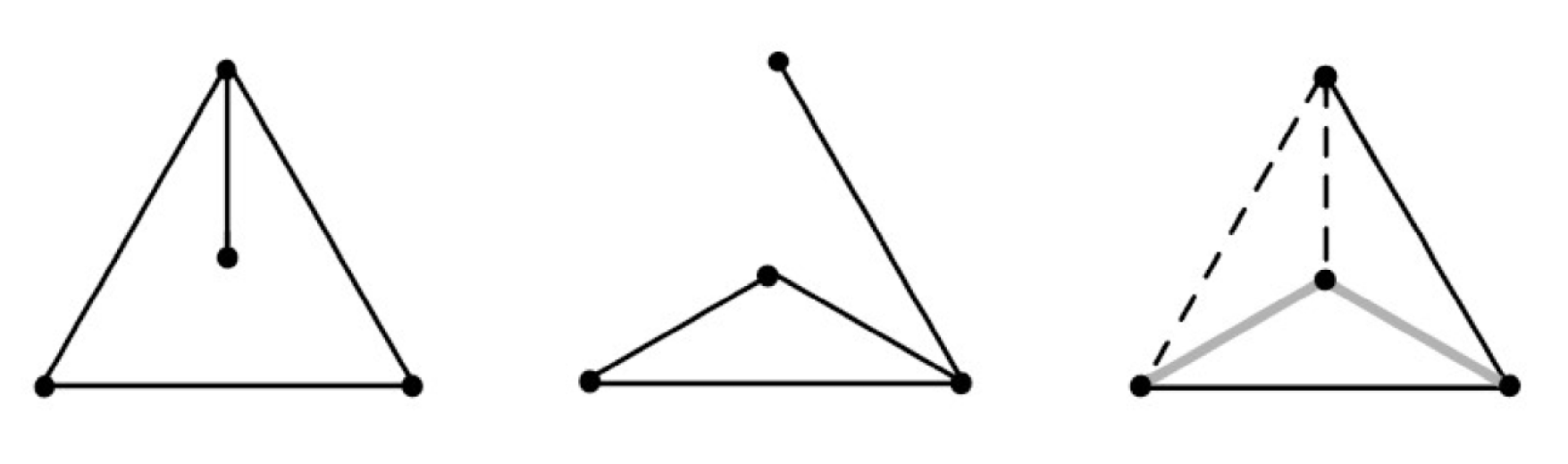}
\put(5.5,16){$\infty$}
\put(13,18){$\infty$}
\put(20.5,16){$5$}
\put(43,10){$2$}
\put(54,10){$2$}
\put(49,4){$3$}
\put(56.1,14){$5$}
\put(77,18){$\infty$}
\put(89,18){$5$}
\put(82.9,20){$\infty$}
\put(79.5,10){$2$}
\put(88,10){$2$}
\put(83.5,4){$3$}
\end{overpic}
    \caption{From left to right: the Coxeter--Dynkin diagram, the defining graph, and the complete Coxeter graph for the system giving the Coxeter group $W=\spans{ r,s,t,u \mid r^2, s^2, t^2, u^2, (rs)^3, (st)^5, (ru)^2, (su)^2 }$.}
\end{figure}
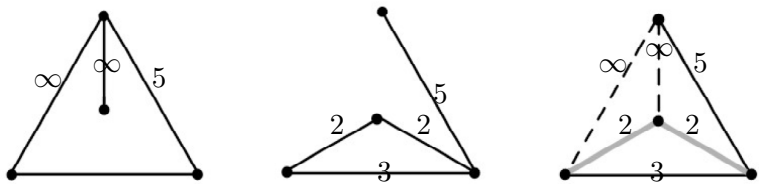

The term `diagram' is used interchangeably with `graph' when speaking about any such graphs. We constantly interchange between Coxeter systems and their complete Coxeter graphs. Any concept introduced for a system (e.g., sphericity) is also interpreted as a concept for the diagram and vice-versa. 

We say that a Coxeter system and the associated Coxeter group is \emph{reducible} if its Coxeter--Dynkin diagram is disconnected. Otherwise it is called \emph{irreducible}.
Recall that a group $G$ is called \emph{directly decomposable} if there exist nontrivial  $H_1, H_2 \leq G$ such that $G \cong H_1 \times H_2$; otherwise it is called \emph{directly indecomposable}. 
The following well-known result (cf. \cite[Chapter~VI]{BourbakiLie4-6}, \cite{MihalikRatcliffeTschantz}, or~\cite[Theorem~3.3]{NuidaIndecomp}) characterizes directly decomposable Coxeter groups. 

\begin{thm}[{Few irreducibles decompose}] \label{thm:NuidaDirect}
Let $(W,S)$ be an {irreducible} Coxeter system. Then $W$ is directly decomposable if and only if $(W,S)$ is of type $\tB_{2k+1}$ or $\tI_2(4k+2)$, for some $k \geq 1$, or of type $\tE_7$ or $\tH_3$. 
\end{thm}

Note that an infinite Coxeter group is reducible if and only if it is directly decomposable.

Similarly, decompositions of Coxeter groups into free products are well understood. 
Recall that $G$ \emph{decomposes} (or \emph{splits}) \emph{as an amalgam} of $A$ and $B$, amalgamated along $C$, if there exist subgroups $A, B, C \leq G$ with $A \neq 1 \neq B$ and $C \leq A \cap B$ such that $G \cong A \ast_C B$. The following was shown in~\cite{MihalikTschantzVisual} using Bass--Serre theory. 

\begin{thm}[{Splittings along Coxeter subgroups}] \label{thm:MihalikTschantzVisual}
A Coxeter group $W$ splits as an amalgam along a {finite} subgroup if and only if there exist nontrivial parabolic subgroups $W_1, W_2 \leq W$ and a common {spherical} parabolic $P \leq W_1 \cap W_2$ such that $W \cong W_1 \ast_{P} W_2$.
\end{thm}

A group $G$ is said to be \emph{freely decomposable} if it can be written as a free product of two subgroups, otherwise $G$ is called freely \emph{indecomposable}.

\begin{cor}[{Free decompositions}]
A Coxeter group is freely decomposable if and only if it is a free product of Coxeter subgroups.
\end{cor}

As elucidated above, the structure of Coxeter groups is very much constrained by its parabolic subgroups. A particularly powerful tool to compare Coxeter systems is given by the {matching theorems}. These were established by Mihalik, Ratcliffe and Tschantz~\cite{MihalikRatcliffeTschantz} drawing from classical results~\cite{BourbakiLie4-6} and the work of Deodhar~\cite{Deodhar}.  

\begin{dfn}[{Basic subsets}] \label{def:basicsubsets}
    Given a Coxeter system $(W,S)$, a subset $B \subseteq S$ is said to be \emph{basic} if $\spans{B}$ is a maximal irreducible {noncyclic} spherical parabolic subgroup. (That is, $(\spans{B},B)$ is irreducible with $2 < \vert \spans{B}\vert < \infty$.)
\end{dfn}

\begin{thm}[{Basic matchings~\cite[Theorem~4.18]{MihalikRatcliffeTschantz}}] \label{thm:basicmatchingthm}
Let $(W_1,S_1)$ and $(W_2,S_2)$ be Coxeter systems with $W_1 \cong W_2$. Then, for each $B_1 \subseteq S_1$ basic, there exist a basic subset $B_2 \subseteq S_2$ and an embedding $f:\spans{B_i} \into \spans{B_j}$ (where $i,j \in \{1,2\}$, $i \neq j$, and $\vert \spans{B_i}\vert \leq \vert \spans{B_j}\vert$) with the following properties. 
\begin{enumerate}
    \item In case $\spans{B_1}$ and $\spans{B_2}$ have the same order, then $(\spans{B_1},B_1)$ and $(\spans{B_2},B_2)$ have the same type, and the map $f$ is an isomorphism. 
    (Here we say that $B_1$ and $B_2$ \emph{match isomorphically}.)
    \item If $\spans{B_1}$ and $\spans{B_2}$ do not have the same order, $\vert \spans{B_1}\vert < \vert \spans{B_2}\vert$ say, then $f: \spans{B_1} \into \spans{B_2}$ and there is some $k \geq 1$ such that either $\spans{B_1}$ is of type $\tD_{2k+1}$ and $\spans{B_2}$ is of type $\tB_{2k+1}$, or $\spans{B_1}$ is of type $\tI_2(2k+1)$ and $\spans{B_2}$ is of type $\tI_2(4k+2)$.
\end{enumerate}
Moreover, the basic subset $B_2 \subseteq S_2$ with the above properties is unique, and in both cases~(i) and~(ii) the (images of the) commutator subgroups $[\spans{B_1},\spans{B_1}]$ and $[\spans{B_2},\spans{B_2}]$ are conjugate in $W_1 \cong W_2$, whichever applicable. 
\end{thm}

\cref{thm:basicmatchingthm} has a series of consequences for Coxeter matrices and ranks of Coxeter systems; see, e.g., \cite[Theorems~7.7 and~9.1]{MihalikRatcliffeTschantz} and \cite{MihalikRatcliffe-Ranks}. 
In the next section, we introduce a space encoding complete Coxeter graphs, and we shall translate some of those results into statements about that space.

%%%%%%%%%%%%%%%%%%%%%%%%%%%%%%%%%%%%%%%%%%%%%%%%%%
% Part 1: Coxeter galaxy
%%%%%%%%%%%%%%%%%%%%%%%%%%%%%%%%%%%%%%%%%%%%%%%%%%%

\section{The galaxy of Coxeter groups}
\label{sec:galaxy}

\textit{A galaxy is a massive, gravitationally bound system that consists of stars, stellar objects, [...], black holes, and an unknown component of dark matter.\footnote{ From \url{https://space-facts.com/galaxies/}, accessed July 22nd 2022.} }

\medskip

In order to study the isomorphism problem for Coxeter groups from a slightly new perspective we now introduce a simplicial complex encoding complete Coxeter graphs and their isomorphism types. 
Some regions of this space are well understood and the isomorphism problem is solved in those parts. However, some regions and pieces of its structure remain mysterious.

%%%%%%%%%%%%%%%%%%%%%%%%%%%%%%%%%%%%%%%%%%%%%%%%%%%%%%%
%%%%%%%%%%%%%%%%%%%%%%%%%%%%%%%%%%%%%%%%%%%%%%%%%%%%%%%

\subsection{Definition of the Coxeter galaxy}
\label{subsec:galaxyDefinition}

\begin{dfn}[The Coxeter galaxy] \label{def:Galaxy}
The \emph{Coxeter galaxy} $\mc{G}$ is the flag simplicial complex defined as follows. Vertices in the Coxeter galaxy are complete Coxeter graphs up to (edge-labeled) graph-isomorphism. 
There is an edge between any pair of vertices whenever there exists an 
%(necessarily  non-type preserving) 
isomorphism between their underlying Coxeter groups. 
The higher dimensional simplices of the Coxeter galaxy $\mc{G}$ are spanned by the complete subgraphs of the graph above.
\end{dfn}

It is clear from the definition that the connected components in the Coxeter galaxy are (potentially infinite dimensional) simplices. Compare \cref{thm:FinitenessConnComps} for more on their size and shape. \cref{exm:simplex-in-galaxy} provides a concrete example of a simplex in the galaxy.

\begin{figure}[htb]
	\centering
	\begin{overpic}[width=0.4\textwidth]{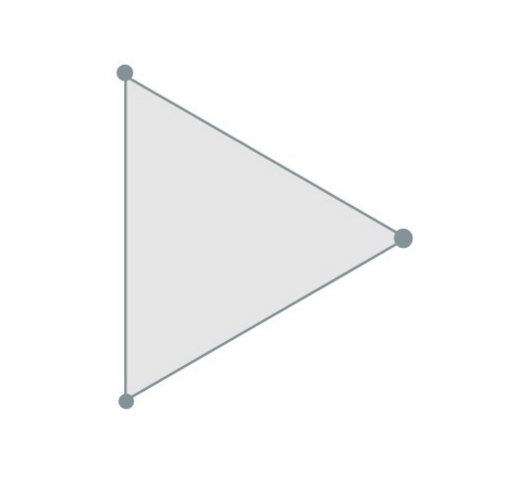}
		\put(10,15){$\CCGo_0$}
		\put(85,45){$\CCGo_1$}
		\put(10,80){$\CCGo_3$}
	\end{overpic}    
	\caption{A $2$-simplex in the Coxeter galaxy. For further details see \cref{exm:simplex-in-galaxy}. The graphs labeling the vertices can be found in \cref{fig:isomorphicGroups}.}
	\label{fig:example_simplexinthegalaxy}
\end{figure}

\begin{remark}[{Dark matter\footnote{From \url{https://en.wikipedia.org/wiki/Dark_matter}, accessed November 24th 2022: ``\textit{Dark matter is called ``dark'' because it does not appear to interact with the electromagnetic field, [...] and is, therefore, difficult to detect.}''} in the galaxy, and isolated points}] \label{rem:invisibleinfo} 
Some information about Coxeter groups cannot be seen in the Coxeter galaxy. 
Recall that two edge-labeled graphs are isomorphic when there exists a bijective map from one vertex set to the other that preserves adjacency and labels.
This implies:
\begin{enumerate}
    \item Two distinct vertices in the Coxeter galaxy never correspond to Coxeter groups of the same Coxeter type. (In particular, an edge always corresponds to an (abstract) non-type preserving isomorphism between the underlying Coxeter groups.)
    \item Since any group automorphism $\alpha \in \Aut(W)$ preserves orders, it canonically induces a graph isomorphism between $\CCG{W,S}$ and $\CCG{W,\alpha(S)}$, so that the systems $(W,S)$ and $(W,\alpha(S))$ are represented by the same vertex in the Coxeter galaxy $\mc{G}$. (In particular, if a system $(W,S)$ is given representing the vertex $\CCGo$ in $\mc{G}$, we may replace $(W,S)$ by any of its conjugates $(W,S^w)$ for $w \in W$.)
\end{enumerate}
Thus, a Coxeter group is \emph{rigid} (in the sense of \cite{BMMN}) if and only if it corresponds to an isolated point in the Coxeter galaxy. While a full characterization of rigid Coxeter systems in terms of their diagrams is not known, the weaker notion of strong rigidity --- which takes into account only inner automorphisms --- is now completely understood by work of Howlett--M\"uhlherr--Nuida \cite{HowlettMuehlherrNuida}; see \cref{sec:rigiditynotions} for more on rigidity concepts. 
\end{remark}

We will now organize the Coxeter galaxy into layers according to the number of vertices in the complete Coxeter graph.

\begin{dfn}[Layer] \label{def:Layer}
We say that a vertex $\CCG{W,S}$ in the Coxeter galaxy $\mc{G}$ lies \emph{in layer $k \in \Z_{\geq 0}$} if $S$ has $k$ elements. (By vacuity, the trivial group is seen as a Coxeter group on an empty generating set.) 
We let $\mc{G}_{\leq k}$ denote the subcomplex of $\mc{G}$ 
spanned by all vertices in layer at most $k$. ($\mc{G}_{\leq 0}$ consists of a single point.) The \emph{layer $k$ of the Coxeter galaxy}, denoted by $\mc{G}_k$, is its subcomplex spanned by the vertices of $\mc{G}_{\leq k} \setminus \mc{G}_{\leq k-1}$.
\end{dfn}

The concept of layers gives us an easy, visual way of sorting the Coxeter galaxy. We think of each layer $\mc{G}_k$ as sitting at a fixed `height' $k$, so that its elements are depicted horizontally.  This also motivates the following definition.

\begin{dfn}[Vertical and horizontal edges]
An edge $\{\CCGo_1, \CCGo_2\}$ in the Coxeter galaxy is called \emph{horizontal} if the graphs $\CCGo_1$ and $\CCGo_2$ are both contained in the same layer, and \emph{vertical of height $k$} if $\CCGo_1$ and $\CCGo_2$ are in different layers $\mc{G}_{l_1}$ and $\mc{G}_{l_2}$, respectively, with $\vert l_1 - l_2\vert = k>0$. We call an edge \emph{vertical} if it is vertical of height $k$ for some $k$. 
\end{dfn}

\begin{exm}\label{exm:simplex-in-galaxy}
\cref{fig:example_simplexinthegalaxy} shows a 2-simplex in $\mc{G}$. The vertices $\CCGo_0, \CCGo_1$ and $\CCGo_3$ are the graphs illustrated in \cref{fig:isomorphicGroups}. They all define the same underlying group and  lie in layers $\mc{G}_3$, $\mc{G}_4$ and $\mc{G}_5$, respectively. All three edges of this simplex are vertical. The edge connecting $\CCGo_0$ with $\CCGo_3$ connects a vertex in layer $3$ with one in layer $5$.     
\end{exm}

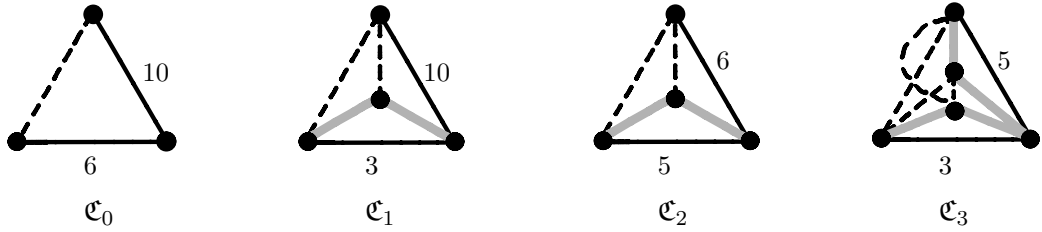
\begin{figure}[htb]
    \centering
    \begin{overpic}[width=\textwidth]{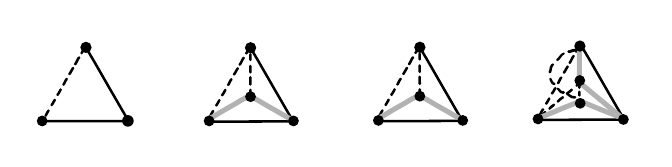}
        \put(12,0){$\CCGo_0$}
        \put(12,4){\small{$6$}} 
        \put(17,12){\small{$10$}}
        \put(36,0){$\CCGo_1$}
        \put(36,4){\small{$3$}} 
        \put(41,12){\small{$10$}}
        \put(61,0){$\CCGo_2$}
        \put(61,4){\small{$5$}} 
        \put(66,13){\small{$6$}}
        \put(85,0){$\CCGo_3$}
        \put(85,4){\small{$3$}} 
        \put(90,13){\small{$5$}}
    \end{overpic}
\caption{Four complete Coxeter graphs describing a same group. The three graphs $\CCGo_i$ with $i=1,2,3$ may be obtained from $\CCGo_0$ (leftmost) by (repeated) applications of blow-ups.}
\label{fig:isomorphicGroups}
\end{figure}

We now introduce a natural subcomplex of the galaxy 
which, as will be shown in \cref{prop:verticalCore}, encodes the connected components of $\mc{G}$. 

\begin{dfn}[Vertical core] \label{def:Core}
	The \emph{vertical core} of the Coxeter galaxy is the subcomplex $\VC$ of $\mc{G}$ obtained by deleting, for every $k > 1$, all vertical edges of height $k$, and all open simplices containing such edges. Considering layers, we define $\VC_{\leq n} := \mc{G}_{\leq n} \cap \VC$ as well as $\VC_n := \mc{G}_n \cap \VC$.
\end{dfn}

Visually, the vertical core is obtained by deleting edges that `jump' between nonsuccessive layers.

\begin{exm}
Let $\CCGo_0$ be a complete Coxeter graph on three vertices $v_1, v_2, v_3$. Suppose the edge $\{v_1,v_2\}$ is labeled by $2(2k_1+1)$ and the edge $\{v_2,v_3\}$ by $2(2k_2+1)$. The remaining third edge has label $\infty$. This graph is an instance of a starlet defined later in \cref{def:starlet}. In case $4k_1+2=6$ and $4k_2+2=10$ the graph is shown on the left in \cref{fig:isomorphicGroups}. It yields the same group as the graphs $\CCGo_1$, $\CCGo_2$ and $\CCGo_3$ depicted in the same figure. 

The connected component of the graph $\CCGo_0$ in the Coxeter galaxy is a three-simplex, as illustrated in \cref{fig:example_verticalcore}, on the four vertices shown in \cref{fig:isomorphicGroups}. The vertical core contains two triangles. One with vertices $\CCGo_0, \CCGo_1, \CCGo_2$ and the other has vertex $\CCGo_3$ instead of $\CCGo_0$. A spine, as defined in \cref{def:spine}, would be given by any tree on the four vertices. For example the following three edges:  $\{\CCGo_0, \CCGo_1\}, \{\CCGo_1,\CCGo_3\}$ and $\{\CCGo_1,\CCGo_2\}$. The first two edges are vertical edges from layer three to four and four to five, respectively. The last edge is a horizontal one within layer four. 
See also \cref{ex:twist-not-enough} for further details on this example. 
\end{exm}

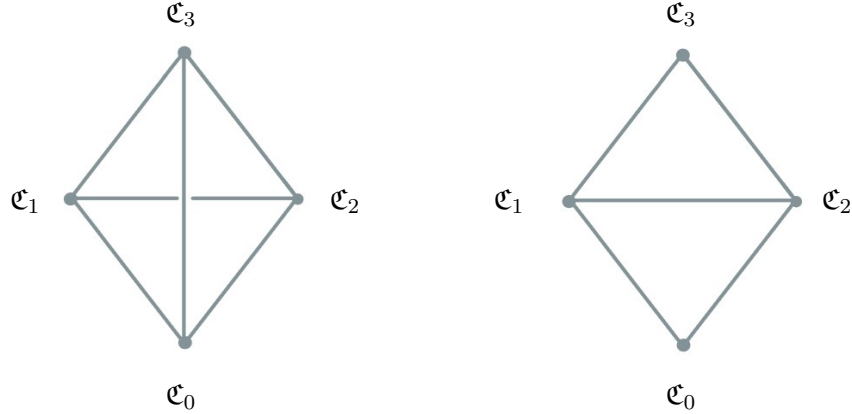
\begin{figure}[htb]
    \begin{overpic}[width=0.7\textwidth]{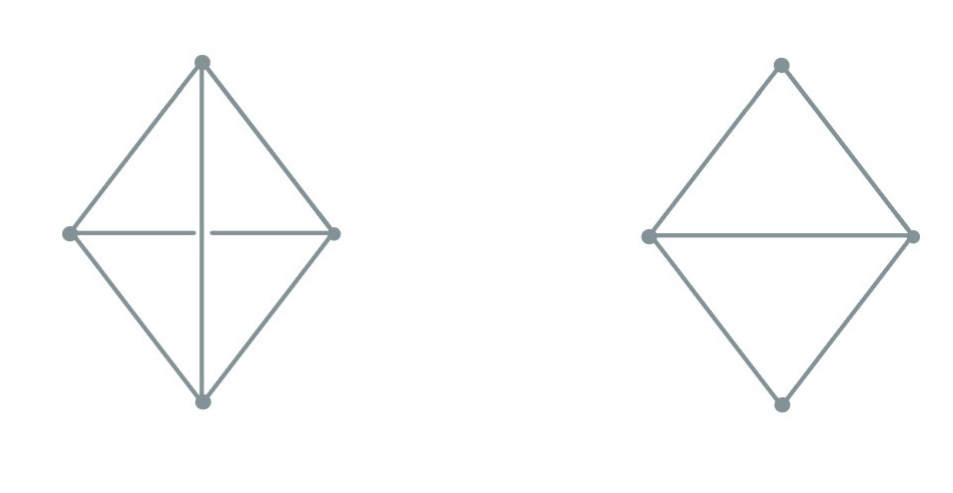}
        \put(19,2){$\CCGo_0$}
        \put(0,26){$\CCGo_1$}
        \put(39,26){$\CCGo_2$}
        \put(19,49){$\CCGo_3$}
        \put(80,2){$\CCGo_0$}
        \put(59,26){$\CCGo_1$}
        \put(99,26){$\CCGo_2$}
        \put(80,49){$\CCGo_3$}
    \end{overpic}
\caption{A tetrahedron ($3$-cell) in the galaxy (left) and its intersection with the vertical core (right), two $2$-cells attached. The graphs labeling the vertices are shown in \cref{fig:isomorphicGroups}.}.
\label{fig:example_verticalcore}
\end{figure}

\begin{remark} \label{obs:GalaxyandVChavesamevertices}
The vertex sets of the Coxeter galaxy $\mc{G}$ and of its vertical core $\VC$ agree as we only delete edges and higher dimensional (open) cells to construct it. In particular, $\mc{G}_n \cap \VC = \mc{G}_n$ for every $n \in \Z_{\geq 0}$.
\end{remark}

There is another notion that can be 
helpful to study algorithmic aspects of the isomorphism problem. 

\begin{dfn}[Spine]
\label{def:spine}
Define a \emph{spine} of the galaxy of Coxeter groups to be a spanning forest of the one-skeleton of the vertical core, {consisting of a single spanning tree for each component}. 
A \emph{spine} of a component is just a spanning tree for it. 
\end{dfn}

A spine, by definition, visits every vertex of a connected component once. One can label or color the edges of a spine by the underlying group isomorphisms or moves; 
cf. \cref{subsec:moves,subsec:horizontalmoves}.

%%%%%%%%%%%%%%%%%%%%%%%%%%%%%%%%%%%%%%%%%%%%%%%%%%%%%%%
%%%%%%%%%%%%%%%%%%%%%%%%%%%%%%%%%%%%%%%%%%%%%%%%%%%

\subsection{The isomorphism problem and related questions} 
\label{subsec:Isoproblem}

The isomorphism problem for Coxeter groups is a classic problem which appeared, together with closely related questions, in several places in the literature. See, for example, Problems~1 and~2 and further discussions in \cite{BernhardSurvey}. In our setting, the isomorphism problem for (finitely generated) Coxeter groups can be formulated as follows.

\begin{qst}[The (classic) isomorphism problem]
\label{qst:isoProblem}
Given two complete Coxeter graphs, is there an algorithm that decides whether the corresponding vertices of the Coxeter galaxy are in a same connected component? 
\end{qst}

A slightly stronger question is the following.  

\begin{qst}[The isomorphism class problem]
\label{qst:refinedGalactic}
Does there exist an algorithm that takes as an input a vertex in the Coxeter galaxy and lists all vertices in its connected component?  
\end{qst}

We also propose the following question which is aimed towards finding an efficient algorithm computing all vertices in a connected component. 
\begin{qst}[The effective isomorphism class problem]
\label{qst:refinedSpine}
Does there exist an algorithm that takes as input a vertex of the Coxeter galaxy and computes a spine (of minimal diameter) of its connected component? 
\end{qst}

%%%%%%%%%%%%%%%%%%%%%%%%%%%%%%%%%%%%%%%%%%%%%%%%%%
% Part 2: New results and more questions
%%%%%%%%%%%%%%%%%%%%%%%%%%%%%%%%%%%%%%%%%%%%%%%%%%%

\section{A first expedition into the galaxy}
\label{sec:FirstExpedition}

In this section we collect general properties of the Coxeter galaxy and the size and shape of its connected components.

\subsection{Vertically navigating the Coxeter galaxy}
\label{subsec:moves}

Classic examples (cf. \cite{NuidaIndecomp}) of abstractly isomorphic Coxeter groups with generating sets of different ranks immediately imply that some edges in the Coxeter galaxy connect vertices in different layers. This phenomenon gives means to vertical moves between different layers which can be fully characterized. 

%%%%%%%%%%%%%%%%%%%%%%%%%%%%%%%%%%%%%%%%%%%%%%%%%%%%
\subsubsection{Blow-ups along pseudo-transpositions}

A first type of isomorphism between Coxeter groups changing the number of generators is a blow-up along a pseudo-transposition. This increases the rank of a Coxeter system by one. 

\begin{dfn} \label{def:pseudotransp}
Given a Coxeter system $(W,S,M_S)$, a Coxeter generator $t \in S$ is called a \emph{pseudo-transposition} 
when the following conditions are satisfied. 
\begin{enumerate}
\item The generator $t$ is contained in a subset $J \subseteq S$ for which any $s \in S\setminus J$ either has  $\mst=\infty$ or satisfies $m_{s,u}=2$ for all $u \in J$. \item The parabolic subgroup $W_J$ is of classical type $\tB_{k}$ for some odd $k$ or of type $\tI_2(2k)$ for some odd $k > 2$.
\item In case $W_J$ is of type $\tB_k$, the element $t$ must also satisfy $m_{t,u}=2$ for all $u \in J$ except for a single $v \in J$ for which $m_{t,v}=4$. (Pictorially, $t$ is represented by the external vertex of the unique $4$-labeled edge of the Coxeter--Dynkin diagram $\tB_k$.)
\end{enumerate}
In case the Coxeter system $(W,S,M_S)$ does not admit any pseudo-transpositions, then we call it \emph{saturated}.\footnote{We should mention that saturated Coxeter systems were introduced as \emph{reduced} systems in~\cite{HowlettMuehlherr}, see also Section 3 of \cite{BernhardSurvey}. They were later referred to as \emph{expanded} systems in~\cite{RatcliffeTschantz}. Bernhard M\"uhlherr suggested to use the term \emph{saturated}. We are convinced that this name fits best with the heuristic meaning that a saturated system is loaded with as many reflections in the generating set as possible. 
This terminology also avoids potential confusion with `reducibility' concepts, such as those from \cref{sec:Prelim}.}
\end{dfn}

Notice that if the parabolic $W_J$ is of type $\tI_2(2k)$ in condition~(ii) above, then $J = \Menge{t,v}$ and $m_{t,v}\neq 2$. In particular, regardless of whether $W_J$ is of type $\tB_k$ or $\tI_2(2k)$, there is always a unique element $v \in J$ which does not commute with $t$.

That pseudo-transpositions are a source of moves to transform Coxeter systems without transforming the underlying Coxeter group is summarized in the following result.

\begin{lem}[{Blow-ups \cite[Theorems~8.4 and~8.8]{MihalikRatcliffeTschantz}}] \label{lem:blowingup}
Suppose a Coxeter system $(W,S,M_S)$ contains a pseudo-transposition $t \in J \subseteq S$ as in \cref{def:pseudotransp}, with $v \in S$ being the unique element of $J$ for which $m_{t,v} > 2$. Then the set 
\[S' := \Menge{tvt,w_J} \cup (S\setminus\Menge{t}),\]
where $w_J$ is the longest element of the parabolic $W_J$, is a Coxeter generating set for $W$ with $\vert S'\vert = \vert S\vert +1$. Moreover, the Coxeter matrix $M_{S'}$ has entries $(m'_{x,y})$ as follows. 
\begin{equation*}
m'_{x,y}=\begin{cases}
          \frac{m_{t,v}}{2} \quad &\text{if } \, x=tvt \text{ and } y=v, \\
          m_{x,y}    \quad &\text{if } \, x,y\in S\cap S', \\
          m_{v,y}  \quad &\text{if } \, x=tvt \text{ and } y \in J \setminus\Menge{t,v},\\
          2 \quad &\text{if } \, x = w_J \text{ and } y \in (\Menge{tvt} \cup J) \setminus \Menge{t},\\
\infty \quad & \text{if } \, x \in \Menge{tvt,w_J} \text{ and } y \in S \setminus J  \text{ and } m_{t,y}=\infty, \\
2 \quad & \text{if } \, x \in \Menge{tvt,w_J} \text{ and } y \in S \setminus J  \text{ and } m_{t,y}\neq\infty.
     \end{cases}
\end{equation*}
\end{lem}

In particular, the unique element $v \in J$ not commuting with the pseudo-transposition $t \in J$ is still contained in the new Coxeter generating set $S'$. (In fact, $S \setminus J \subseteq S'$.) The pseudo-transposition $t$ gets replaced by the reflection $tvt$, and the longest element $w_J \in \spans{J}$ gets `promoted' to a reflection in the new system $S'$.

While \cref{lem:blowingup} is not new, we briefly sketch a proof for the sake of completeness.

\begin{proof}
The pseudo-transposition $t$ that has been dropped out of the generating set is contained in the subgroup generated by $tvt^{-1}$, $w_J$, and $J \setminus \{t\}$. Hence $S'$ is still a generating set for $W$. 

The fact that the relations among the elements of $S'$ are Coxeter relations as in \cref{def:CoxeterSystem} follows from standard results (see, e.g., \cite{DavisBook}). Indeed, most of them can be read off of the original Coxeter system $(W,S,M_S)$, namely for elements in $S\cap S'$. The relations $(tvt w_J)^2=1=(w_J y)^2$ for $y \in S'\cap J$ are clear since $w_J$ is central in $W_J$, and it is immediate that $tvt v=(tv)^2$ has order $m_{t,v}/2$. By \cref{def:pseudotransp}, $v$ is the unique Coxeter generator in $J$ not commuting with $t$, so that $(tvt y)^n = t(vy)^n t$ for any $y \in J \setminus\Menge{v}$, hence $tvt y$ has order $m_{v,y}$. Now, if some $y \in (S \cap S') \setminus J$ satisfies $m_{t,y}=\infty$, then $W$ decomposes as the amalgamated free product of its (maximal) parabolics $W_{S\setminus\Menge{t}}$ and $W_{S\setminus\Menge{y}}$, amalgamated along the common parabolic $W_{S\setminus\Menge{t,y}}$. Hence, since $t \notin S'$ and both $tvt$ and $w_J$ include $t$ in their reduced expressions, the orders of $tvt y$ and of $w_J y$ are infinite. In case $m_{t,y} < \infty$, then the definition of pseudo-transposition implies that $y$ commutes with all elements of $J$, whence $m_{tvt,y}=2=m_{w_J,y}$. 

Thus the abstract Coxeter group defined by (a copy of) the generating set $S'$ and the Coxeter matrix $M_{S'}$, obtained from the computations in the previous paragraph, maps onto the subgroup $\spans{S'} \leq W$. It is straightforward to check that such Coxeter group injects into $W$, so that it is indeed isomorphic to $\spans{S'}$. 
\end{proof}

In view of \cref{def:pseudotransp} and \cref{lem:blowingup}, we interchangeably say that $\CCG{W,S}$ can be blown up if $(W,S)$ admits a pseudo-transposition. 

\begin{exm} 
Let us illustrate how blow-ups modify the underlying diagrams. For the sake of tradition, we work this example out using the Coxeter--Dynkin diagram (cf. \cref{def:CoxDiag}). Consider the spherical Coxeter group whose classical type is the disjoint union of $\tB_5$ and $\tI_2(6)$, so that this group is the direct product of the hyperoctahedral group of order $2^5 5!$ and the dihedral group of order $12$. It represents a vertex in $\mc{G}$ shown as the bottom corner of \cref{fig:simplex-new}.

One quickly checks (cf. \cref{def:pseudotransp}) that this system admits two blow-ups, one along the subgroup of type $\tB_5$, and another along the subgroup of type $\tI_2(6)$. The top right diagram of \cref{fig:simplex-new} shows the blown up diagram along the old edge of type $\tI_2(6)$, in which case the added vertex is the longest element from $\tI_2(6)$, now `promoted' to a reflection, and the dihedral order gets halved. The top left graph in \cref{fig:simplex-new} represents the diagram blown up along the parabolic subgroup of type $\tB_5$, which gets replaced by the diagram of type $\tD_5$, and the added vertex is the longest element from $\tB_5$. By \cref{def:Galaxy}, the three diagrams span a $2$-simplex in the Coxeter galaxy.
\end{exm}

\begin{figure}[htb]
    \centering
    \begin{overpic}[width=0.5\textwidth]{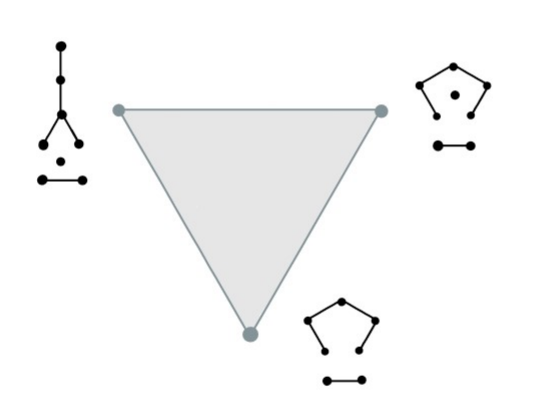}
        \put(63,0){\small{$6$}}
        \put(55,11){\small{$4$}}
        \put(11,38){\small{$6$}}
        \put(84,44){\small{$3$}}
        \put(76,55){\small{$4$}}
    \end{overpic}
    \caption{An example of how blow-ups modify Coxeter--Dynkin diagrams, starting with the disjoint union of $\tI_2(6)$ and $\tB_5$ (bottom vertex).}
    \label{fig:simplex-new}
\end{figure}
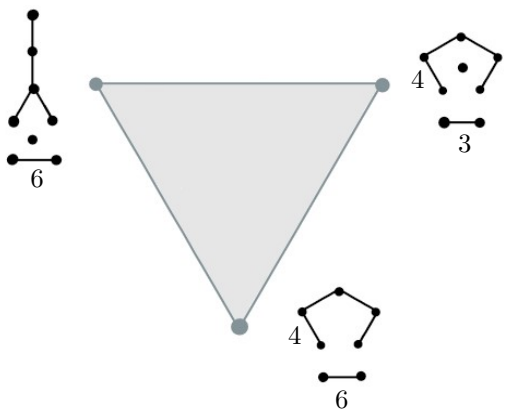

Drawing from the work of Mihalik--Ratcliffe--Tschantz~\cite{MihalikRatcliffeTschantz} one can see that blow-ups are `atomic vertical edges', in the sense that they suffice to describe vertical navigation within the Galaxy.

\begin{thm}[{Vertical edges decompose into blow-ups~\cite{MihalikRatcliffeTschantz}}] 
\label{thm:navigatingverticallayers}
Let $\CCGo_1$, $\CCGo_2$ be adjacent vertices in $\mc{G}$ lying in layers $\mc{G}_k$ and $\mc{G}_{k+m}$, respectively, with $m\geq 1$. Then $\CCGo_1$ admits a blow-up and there is a vertex $\CCGo_3$ in $\mc{G}_{k+1}$ obtained from $\CCGo_1$ via the corresponding pseudo-transposition, as described in \cref{lem:blowingup}, so that $\CCGo_1$, $\CCGo_2$, $\CCGo_3$ span a $2$-simplex in $\mc{G}$.
\end{thm}

\begin{proof}
Write $\CCGo_1 = \CCG{W,S}$, where $\vert S\vert = k$. By the Maximum Rank Theorem~\cite[Theorem~9.1]{MihalikRatcliffeTschantz}, $(W,S)$ must admit some blow-up, for otherwise $\vert S\vert$ would be maximal among cardinalities of Coxeter generating sets for $W$. That is, $S$ contains a subset $J \subseteq S$ spanning a maximal spherical parabolic of classical type $\tI_2(4n+2)$ or $\tB_{2n+1}$ for some $n \geq 1$ and satisfying the conditions from \cref{def:pseudotransp}. By \cref{lem:blowingup}, we can construct a Coxeter generating set $S' \subseteq W$ for $W$ with $\vert S'\vert = k+1$. Thus $\CCGo_3:=\CCG{W,S'} \in \mc{G}_{k+1}$ is adjacent to both $\CCGo_1$ and $\CCGo_2$, whence the theorem.
\end{proof}

Closely related to the above is the fact that a nonsaturated Coxeter system always admits pseudo-transpositions such that the result after finitely many blow-ups is saturated. This was shown as \cite[Proposition~7]{HowlettMuehlherr}, see also \cite[p.~2271]{CapracePrzytyckiTwistRigid} or \cite[Section~9]{MihalikRatcliffeTschantz}.

\begin{remark} \label{obs:maxranknoblowups}
It is useful to keep in mind the following fact, contained in~\cite[Theorem~9.1]{MihalikRatcliffeTschantz} and used in the proof of \cref{thm:navigatingverticallayers}: Given a Coxeter system $(W,S)$, the set $S$ has maximal cardinality among Coxeter generating sets of $W$ if and only if $S$ does not admit a blow-up (i.e., $S$ contains no pseudo-transpositions).
\end{remark}

\cref{thm:navigatingverticallayers} has the following straightforward, but structurally important consequences.

\begin{cor}[{Navigating up}] \label{cor:navigatingverticallayers}
Every vertical {path} in the Coxeter galaxy $\mc{G}$ from layer $k$ to layer $k+l$ can be {homotoped to} a sequence of $l$ vertical edges between adjacent layers, potentially followed by a single horizontal edge within layer $k+l$. All vertical edges in this sequence correspond to blow-ups. 
\end{cor}

\begin{cor}[{Connected components do not skip layers}] \label{cor:componentsdonotjump} 
Any connected component $\mc{C}$ of the Coxeter galaxy $\mc{G}$ containing a vertex in some layer $k$ and a vertex in some layer $k+m$ with $m > 1$ also contains vertices in all layers $\mc{G}_{k+i}$ with $0 < i < m$.
\end{cor}

The results discussed in this section explain how to move to a higher layer. The next subsection will consider moves to a lower layer of the galaxy.

%%%%%%%%%%%%%%%%%%%%%%%%%%%%%%%%%%%%%%%%%%%%%%%%%%%%
%%%%%%%%%%%%%%%%%%%%%%%%%%%%%%%%%%%%%%%%%%%%%%%%%%%%
%%%%%%%%%%%%%%%%%%%%%%%%%%%%%%%%%%%%%%%%%%%%%%%%%%%%
\subsubsection{Blowing down}

A second type of isomorphism between Coxeter groups that alters the number of generators is a blow-down. This move was implicitly introduced by Mihalik; cf. \cite[Section~3]{MihalikEvenIsoThm}. Given a Coxeter system $(W,S)$ and a subset of generators $X \subseteq S$, we let $X^\perp$ denote the set of all elements $s \in S$ for which $m_{s,x}=2$ for all $x \in X$. 

\begin{thm}[{Blow-downs\cite{MihalikRatcliffe-Ranks}}] \label{thm:blowdown}
Let $(W,S)$ be a Coxeter system. There exists a Coxeter generating set $S'\subseteq W$ with $\vert S'\vert < \vert S\vert$ if and only if all of the following hold.
\begin{enumerate}
    \item There is a basic subset $B \subseteq S$ (cf. \cref{def:basicsubsets}) with $(\spans{B},B)$ 
    of classical type $\tD_{2n+1}$ or $\tI_2(2n+1)$ (with $n\geq 1$) and an element $r \in B^\perp$ such that each generator $s \in S$ for which $m_{s,r} < \infty$ lies in $B \cup B^\perp$ and satisfies $m_{s,r} = 2$.
    \item In case $(\spans{B},B)$ is of type $\tD_{2n+1}$, the two vertices $x,y$ at the split end of the Coxeter--Dynkin diagram for $(\spans{B},B)$ satisfy $\{ s \in S \mid m_{s,x}, m_{s,y} < \infty \} = B \cup B^\perp$. Similarly, if $B = \{x,y\}$ (that is, $(\spans{B},B)$ is of type $\tI_2(2n+1)$), then $\{ s \in S \mid m_{s,x}, m_{s,y} < \infty \} = B \cup B^\perp$.
    \item In both cases above, if there is an edge cycle in $\CCG{W,S}$ of length at least four containing $\{x,y\}$ and with no $\infty$-labeled edges, then such a cycle must have a finite-labeled chord, that is, there exist two nonconsecutive vertices in this cycle whose label in $\CCG{W,S}$ is finite. 
\end{enumerate}
\end{thm}

The relationship between blow-downs and blow-ups is summarized in the following, a proof of which can be deduced from the proof of~\cite[Theorem~4.1]{MihalikRatcliffe-Ranks}.

\begin{thm}[{Moving down in the galaxy\ ~\cite{MihalikRatcliffe-Ranks}}] \label{thm:downwardsnavigation}
A vertex $\CCG{W,S}$ in the Coxeter galaxy $\mc{G}$ has a neighbor in a lower layer $\mc{G}_{n-m}$ with $n > m > 0$ if and only if it has a neighbor $\CCG{W,S'} \in \mc{G}_{n-1}$ that admits a pseudo-transposition (\cref{def:pseudotransp}), along which $(W,S')$ can be blown-up as in \cref{lem:blowingup}.
\end{thm}

We stress that the statement above does not imply that a blow-down is the exact inverse move of a blow-up. For instance, $\CCG{W,S}$ might first connect to some other point $\CCG{W,S''} \in \mc{G}_{\vert S'\vert}$ without pseudo-transpositions and one might need an intermediate horizontal move connecting $\CCG{W,S''}$ to the neighbor $\CCG{W,S'}$ which does exhibit a pseudo-transposition. This further motivates the investigation of horizontal moves, which will be addressed in more detail in \cref{subsec:horizontalmoves}.

%%%%%%%%%%%%%%%%%%%%%%%%%%%%%%%%%%%%%%%%%%%%%%%%%%%%
%%%%%%%%%%%%%%%%%%%%%%%%%%%%%%%%%%%%%%%%%%%%%%%%%%%%
\subsection{Shape and connected components of the galaxy}
\label{subsec:components}

This section concerns the connected components of the Coxeter galaxy and their `shapes'. Though it is easy to check that the galaxy is infinite dimensional, we prove that each of its components is finite dimensional and provide an upper bound on the number of vertices in a component. 
In order to compute the bound we need some statistics on complete Coxeter graphs. 

\begin{dfn}[Statistics on complete Coxeter graphs.] \label{def:statistics}
Let $\CCGo=\CCG{W,S}$ be a complete Coxeter graph and let, by definition,   
\begin{enumerate}
    \item $u_\CCGo$ be the number of distinct blow-ups the graph $\CCGo$ admits, 
    \item $d_\CCGo$ be the number of distinct blow-downs the graph $\CCGo$ admits, and 
    \item $p_\CCGo$ be the number of distinct standard parabolic subgroups of $W$ of types $\tB_{2n+1}$, $\tD_{2n+1}$, $\tI_2({4n+2})$ or $\tI_2(2n+1)$ {for $n\geq 1$} . 
\end{enumerate}
\end{dfn}

A first observation is that the Galaxy contains arbitrarily long (vertical) paths. This is a direct consequence of the next example. 

\begin{exm}
For every $k \geq 1$ and $n \geq 1$, let $W$ be the $n$-fold direct product of the group of type $\tI_2(4k+2)$. In each of the $n$ copies of $\tI_2(4k+2)$ one can find a pseudo-transposition (cf. \cref{def:pseudotransp}), so that the given group can be blown-up. In total $n$ distinct blow-ups are possible. Notice that each blow-up of the original graph yields the same vertex in the next  layer as the resulting complete Coxeter graphs (after having applied any one of the available blow-ups) are isomorphic.
\end{exm}

The reader might rightfully protest that the example above is uninteresting in the sense that it uses a (spherical) reducible Coxeter group. But it is possible to provide (infinite) irreducible examples which produce arbitrarily big simplices.

\begin{dfn}[Starlets\footnote{We chose this name since the corresponding presentation diagram looks like a little star centered around $s_0$.}]
\label{def:starlet}
    A \emph{starlet} is a complete (simplicial) edge-labeled graph $\CCGo$ on a vertex set $S = \Menge{s_0,s_1,\ldots,s_{n-1}}$ which is constructed as follows. Take $k_1,\ldots,k_{n-1}$ to be pairwise distinct positive integers. Connect $s_0$ to each $s_i$, $i\neq 0$, with an edge labeled $4k_i+2$, and connect all other pairs $s_i$, $s_j$ with $i\neq j$ by an edge labeled $\infty$. 
\end{dfn}

\begin{lem}[{Little stars in the galaxy}] \label{lem:starlet}
Let $n > 1$ be arbitrary.
\begin{enumerate}
    \item A starlet on $\vert S\vert =n$ vertices is an irreducible complete Coxeter graph $\CCG{W,S} \in \mc{G}_n$, which is an isolated point in $\mc{G}_{\leq n}$ and allows for exactly $n-1$ distinct blow-ups. (That is, $u_{\CCG{W,S}} = n-1$.) Moreover, $W$ is spherical if and only if $n=2$. 
    \item The connected component $\mc{C}$ of a starlet intersects nontrivially with all layers $\mc{G}_{n+k}$ with $k\in\{0,1,\ldots, n-1\}$. 
    \item The component $\mc{C}$ contains a simplex $\sigma_{\mrm{max}}$ of dimension $\sum_{i=1}^{n-1}\binom{n-1}{i}$ and, for each $i = 0, \ldots, n-2$, a horizontal simplex $\sigma_i \subseteq \mc{G}_{n+i}$ of dimension $\binom{n-1}{i}-1$.
\end{enumerate}
\end{lem}

\begin{proof}
Let $\CCGo$ be a starlet. 
The Coxeter system $(W,S)$ with $\vert S \vert = n$ having $\CCGo$ as its complete Coxeter graph is \emph{even}, i.e., all of its labels are either even positive integers or $\infty$. The fact that $W$ is infinite if and only if $n > 2$ is immediate from the definition. By work of Bahls~\cite{BahlsThesis} and Mihalik~\cite{MihalikEvenIsoThm}, even Coxeter groups have {unique} even Coxeter systems, which cannot be blown down by \cref{thm:blowdown}(i). In particular, the even presentation has smallest possible rank among all Coxeter presentations for the given group, again by \cref{thm:blowdown}. Moreover, such a system cannot admit a horizontal neighbor. Indeed, by uniqueness of even presentations such a potential neighbor would have to be a noneven system for the given Coxeter group, which (by \cite[Theorem~1]{MihalikEvenIsoThm}) could be transformed into a complete Coxeter graph for the same group with one less vertex, contradicting minimality of the rank of the unique even presentation. It thus follows that $\CCG{W,S}$ is isolated within $\mc{G}_{\leq n}$. 

By construction, each generator $s_i \neq s_0$ of $W = \spans{S}$ satisfies the conditions of \cref{def:pseudotransp}. Since the corresponding integers $4k_1+2,\ldots,4k_{n-1}+2$ yielding pseudo-transpositions for $(W,S)$ are all distinct, $\CCG{W,S}$ admits exactly $n-1$ blow-ups. Moreover, after applying all such blow-ups successively, the resulting vertex lies in $\mc{G}_{2n-1}$ and has edge-labels contained in the (set) $\Menge{2,\infty,2k_1+1,\ldots,2k_{n-1}+1}$, hence cannot be further blown up; see \cref{obs:maxranknoblowups}.

To obtain the simplices of prescribed dimensions, observe that $\CCG{W,S} \in \mc{G}_{n}$ can be blown up to $\mc{G}_{n+1}$ in $n-1$ distinct manners, one for each pseudo-transposition. Inductively, each neighbor of $\CCG{W,S}$ in layer $\mc{G}_{n+i}$, for $0 \leq i < n-1$, can be connected (upwards) with  $\binom{n-1}{i}$ different neighbors in layer $\mc{G}_{n+i+1}$, one for each distinct available blow-up (that has not been applied, yet). By definition, the $\binom{n-1}{i}$ neighbors of $\CCG{W,S}$ contained in layer $\mc{G}_{n+i}$ form a complete graph and thus span a $(\binom{n-1}{i}-1)$-simplex entirely contained in $\mc{G}_{n+i}$. Taking all such vertices at once, from $\mc{G}_n$ to the top layer $\mc{G}_{2n-1}$, we obtain the simplex $\sigma_{\mrm{max}}$ of dimension $\sum_{i=1}^{n-1}\binom{n-1}{i}$.
\end{proof}

A direct consequence is the following. 

\begin{cor}[{Dimension of the galaxy}] \label{cor:GalaxyandVCareinfinitedimensional}
Both the Coxeter galaxy $\mc{G}$ and its vertical core $\VC$ are infinite dimensional. They contain arbitrarily large simplices all of whose vertices can be taken to be nonspherical irreducible Coxeter systems.
\end{cor}

\begin{cor}[{Bounded heights}] \label{cor:heightisbounded}
Let $\CCGo$ be a vertex in layer $\mc{G}_n$ of $\mc{G}$ (or of $\VC$). The connected component of $\CCGo$ in $\mc{G}$ (resp. in $\VC$) is contained in $\mc{G}_{\leq 2n-1}$ (resp. in $\VC_{\leq 2n-1}$). This bound is optimal in the sense that there exist vertices in $\mc{G}_n$ whose connected components nontrivially intersect $\mc{G}_{2n-1}$.
\end{cor}

\begin{proof}
By \cref{cor:navigatingverticallayers}, any path from $\CCG{W,S} \in \mc{G}_n$ to higher layers can be replaced by a sequence of edges arising from blow-ups. Since (any neighbor of) $\CCG{W,S}$ can have at most $\vert S\vert -1$ pseudo-transpositions, the first claim follows. The fact that the bound is optimal was just proved in \cref{lem:starlet} using starlets.
\end{proof}

We have obtained local vertical bounds on the size of the galaxy. But what about the horizontal behavior? 

\begin{qst}[Diameter]
	\label{qst:diameter}
    Is there an algorithm which computes the diameter (or the dimension) of any given connected component of the Coxeter galaxy or of the vertical core using the graph of any of its vertices as an input?   
\end{qst}

We now explore layers horizontally. Recall from \cref{def:statistics} and \cref{thm:basicmatchingthm} that the number $p_\CCGo$ bounds (from above) the number of neighbors of $\CCGo \in \mc{G}$ having basic subsets (cf. \cref{def:basicsubsets}) that do not isomorphically match basic subsets of $\CCGo$. 

\begin{lem} \label{lem:layersarefinitedimensional}
Any vertex $\CCGo=\CCG{W,S}$ in an arbitrary layer $\mc{G}_k$ of the Coxeter galaxy has finitely many horizontal neighbors. Moreover, the number of neighbors is bounded above by $2^{p_{\CCGo}} \cdot \binom{k}{2}!$ and thus $\mc{G}_k$ is finite dimensional.
\end{lem}

\begin{proof}
Let $\CCGo=\CCG{W,S} \in \mc{G}_k$ be given, where $\vert S \vert = k$. Assume first that $S$ is a maximal Coxeter generating set for $W$. We argue that $\CCG{W,S}$ has only finitely many neighbors in its layer $\mc{G}_{k}$. 
By the Maximum Rank Theorem~\cite[Theorem~9.1]{MihalikRatcliffeTschantz}, one deduces that the basic subsets of $(W,S)$ and of any of its neighbors within $\mc{G}_k$ match isomorphically (in the sense of \cref{thm:basicmatchingthm}). 
It then follows from the Simplex Matching Theorem \cite[Theorem~7.7]{MihalikRatcliffeTschantz} that the multisets of entries of the Coxeter matrix $M_S$ are the same as those of a Coxeter matrix of its neighbors in $\mc{G}_k$. Therefore the complete Coxeter graph of any neighbor of $\CCG{W,S}$ in $\mc{G}_{k}$ is obtained from $\CCG{W,S}$ by permuting labels (though not necessarily vertices!). Since $\CCG{W,S}$ is finite, there can be only finitely many (in fact at most $\binom{k}{2}!$) elements in $\mc{G}_{k}$ adjacent to $\CCG{W,S}$. 

Suppose now that $\CCGo$ is arbitrary. In that case $\CCG{W,S} \in \mc{G}_{k}$ is not necessarily in the highest possible layer and $S$ is not necessarily a maximal Coxeter system for $W$, but we can also argue with basic subsets. \cref{thm:basicmatchingthm} implies that either the basic subsets of $\CCG{W,S}$ isomorphically match those of its neighbors, or the non-matching basic subsets generate irreducible spherical parabolics of prescribed Coxeter types, which depend only on $\CCG{W,S}$. In the former case, we argue as in the last paragraph using the Simplex Matching Theorem~\cite[Theorem~7.7]{MihalikRatcliffeTschantz}. In the latter case, one has by \cref{thm:basicmatchingthm} that nonisomorphic matchings only occur either between types $\tB_{2n+1}$ and $\tD_{2n+1}$ or between types  $\tI_2(4n+2)$ and $\tI_2(2n+1)$ for some $n \geq 1$. In particular, one can directly inspect on $\CCG{W,S}$ whether these types occur. The total number of such parabolics is bounded above by $p_\CCGo$ as defined in \cref{def:statistics}. Thus $\CCG{W,S}$ has at most $\binom{k}{2}!\cdot 2^{p_\CCGo}$ neighbors within $\mc{G}_k$. The lemma follows. 
\end{proof}

We can now establish the following.

\begin{thm}[The galaxy is locally finite] \label{thm:FinitenessConnComps}
Every connected component of the Coxeter galaxy $\mc{G}$ is a finite dimensional simplex. 
\end{thm}

\begin{proof}
The fact that a connected component of any point in $\mc{G}$ is a simplex is immediate from \cref{def:Galaxy}. Let $\CCGo \in \mc{G}$ be an arbitrary vertex, lying in layer $\mc{G}_k$ say. By \cref{cor:heightisbounded}, its connected component is entirely contained in $\mc{G}_{\leq 2k-1}$. By \cref{lem:layersarefinitedimensional}, every vertex in this component has only finitely many neighbors within each layer, whence the theorem.
\end{proof}

\begin{cor}
Every connected component of the vertical core $\VC$ of the Coxeter galaxy is a finite dimensional subcomplex.
\end{cor}

Now we know that the galaxy and its vertical core are arbitrarily big though locally finite. The size of the simplices grows as we move up layers. But how fast? Notice that, by \cref{cor:heightisbounded}, the height of a connected component is linear as a function of the rank of their bottom-most vertices. In contrast, \cref{lem:layersarefinitedimensional} gives a (highly nonoptimal) upper bound on the dimensions of simplices within each layer. 

\begin{qst}
\label{qst:Growth of dimension} 
How big is the Coxeter galaxy or its vertical core up until a fixed layer $n$? How fast do the functions 
$n\mapsto \mathrm{dim}(\mc{G}_n)$, $n\mapsto \mathrm{dim}(\mc{G}_{\leq n})$,  
$n\mapsto \mathrm{dim}(\VC_n)$, and $n\mapsto \mathrm{dim}(\VC_{\leq n})$ grow?
\end{qst}
    
\begin{qst} \label{qst:growth_verticalVShorizontal}
Given a connected component $\mc{C}$ of the Coxeter galaxy or of its vertical core $\VC$, how large is it `horizontally' in comparison to its height? That is, defining the width and the height of $\mc{C}$ as 
\begin{align*} 
\mathrm{wid}(\mc{C}) & := \mathrm{max}\{ \mathrm{dim}(\mc{C} \cap \mc{G}_n) \mid n \in \N\} \quad \text{ and } \\
\mathrm{ht}(\mc{C}) & := \mathrm{max}\{ \vert m - k \vert \, \mid \, \mc{C} \cap \mc{G}_m \neq \leer \neq \mc{C} \cap \mc{G}_k \},
\end{align*}
respectively, what is the ratio $\frac{\mathrm{wid}(\mc{C})}{\mathrm{ht}(\mc{C})}$?

Let $\mc{G}_n$ be the lowest layer that a component $\mathcal{C}$ intersects nontrivially. 
What is the asymptotic behavior of $\mathrm{wid}(\mc{C})$, $\mathrm{ht}(\mc{C})$ or their ratio, when $n$ goes to infinity? \end{qst}

In view of \cref{obs:GalaxyandVChavesamevertices}, the concepts of width and height of a connected component $\mc{C} \subseteq \mc{G}$ coincide with those for $\mc{C} \cap \VC$ (where we replace $\mc{G}_n$ by $\VC_n$ in the definitions). The next proposition implies that the vertex sets of $\mc{C}$ and $\mc{C}\cap\VC$ agree, so that we may restrict ourselves to understanding the vertical core in order to solve the isomorphism problem.

\begin{prop}[Reduction to the core] \label{prop:verticalCore}
Every connected component of the vertical core $\VC$ contains the same set of vertices as the connected component of any of its elements in the galaxy $\mc{G}$. 
\end{prop}

\begin{proof}
Let $\mc{C} \subseteq \VC$ be an arbitrary connected component. The vertex sets of $\VC$ and $\mc{G}$ agree and, by definition of $\VC$, only edges of height strictly greater than one and corresponding higher-dimensional open cells are missing in the vertical core. Hence $\mc{C} \cap \mc{G}_k = \mc{C}$. 

Thus it remains to show that if $\mc{C}$ has a vertex $\CCGo_1\in \mc{G}_k$ and a vertex $\CCGo_2\in \mc{G}_{k+m}$ for $m\geq 2$, then there must be a path entirely contained in $\VC$ connecting $\CCGo_1$ to $\CCGo_2$. But the existence of such a path is guaranteed by \cref{cor:navigatingverticallayers}.
\end{proof}

%%%%%%%%%%%%%%%%%%%%%%%%%%%%%%%%%%%%%%%%%%%%%%%%%%%%
%%%%%%%%%%%%%%%%%%%%%%%%%%%%%%%%%%%%%%%%%%%%%%%%%%%%
\subsection{Profinite completions and the dimension of the first layers}
\label{subsec:rank3}

The goal of this subsection is to paint a complete picture of $\mc{G}_{\leq 3}$, the first three layers of the Coxeter galaxy $\mc{G}$. By vacuity we include the trivial group as the (only) Coxeter group of rank zero. In the layer $\mc{G}_1$ there is a single $1$-generated Coxeter group, namely the cyclic group $C_2$ of order two. The second layer is given by the family of all dihedral groups. The study of the third layer occupies most of this subsection. We recall the following definition. 

\begin{dfn}[Triangle groups]
\label{def:triangle group}
A \emph{triangle Coxeter group}, in the following simply referred to as \emph{triangle group},  is a Coxeter group $W$ having a complete Coxeter graph $\CCG{W,S} \in \mc{G}_3$. Any such group can be given by a Coxeter presentation of the form 
\[\spans{a,b,c \mid a^2, b^2, c^2, (ab)^p, (bc)^q, (ca)^r},\]
where $p = m_{a,b}$, $q = m_{b,c}$, $r = m_{c,a}$, where as usual we interpret $(st)^\infty$ as `no relation between  $s$ and $t$'. 
A triangle group with defining orders $p,q,r \in \N_{\geq 2} \cup \{\infty\}$ as above will be denoted by $\Delta(p,q,r)$. 
\end{dfn}

The triangle groups $\Delta(p,q,r)$ can be realized as reflection groups of the $2$-sphere $\Sph^2$, the Euclidean plane $\R^2$, or the hyperbolic plane $\Hyp^2$, depending on whether the value $\frac{1}{p}+\frac{1}{q}+\frac{1}{r}$ is greater than, equal to, or less than one, respectively. (Here we use the convention $\frac{1}{\infty}=0$.) The multiset $\Menge{p,q,r}$ and the triangle group is called spherical, Euclidean, or hyperbolic accordingly. 
Most triangle groups are hyperbolic and act on the hyperbolic plane with finite covolume, with generators being reflections about hyperbolic lines.

Any triangle group $\Delta(p,q,r)$ contains a subgroup of index two given by the presentation 
\[\Gamma(p,q,r) \cong \spans{x,y,z \mid x^p, y^q, z^r, xyz}.\]
Such a subgroup $\Gamma(p,q,r)$ is sometimes also referred to as an (ordinary) triangle group in the literature \cite{BridsonConderReid}. It is also known as a \emph{von Dyck} group or a \emph{rotation triangle (sub)group}. 
In the standard realization of the ambient triangle Coxeter groups, the rotation subgroup $\Gamma(p,q,r)$ can be interpreted as the subgroup of $\Delta(p,q,r)$ generated by rotations of prescribed orders $p$, $q$, and $r$. 
If the triple $\Menge{p,q,r}$ is hyperbolic, then $\Gamma(p,q,r)$ is a Fuchsian group as it is a discrete subgroup of $\Isom^+(\Hyp^2) \cong \PSL_2(\R)$; we refer the reader to~\cite{KatokFuchsian} for background on such groups and their geometric properties. 

In the sequel we employ profinite techniques in the study of the Coxeter galaxy. Recall that a (topological) group is \emph{profinite} if it is compact, Hausdorff, and totally disconnected. The \emph{profinite completion} $\hut{G}$ of a finitely generated group $G$ is the inverse limit of its (inverse) system of finite quotients; cf.~\cite{RibesZalesskii} for background and standard results on profinite groups. We say that a group $G$ in a family of groups $\mc{F}$ is \emph{profinitely rigid} (in $\mc{F}$) in case $\hut{G} \cong \hut{H}$ implies $G \cong H$ for all $H \in \mc{F}$. 

For an overview on profinite rigidity and in order to gain a better understanding of its relevance in group theory we refer the interested reader to the excellent ICM survey of Reid~\cite{ReidICM}. One motivation to look at finite quotients of Coxeter groups is the following strong form of profinite rigidity for von Dyck triangle groups $\Gamma(p,q,r)$. 

\begin{thm}[{Bridson--Conder--Reid~\cite[Section~8]{BridsonConderReid}}] \label{thm:BCR}
Suppose $p$, $q$, $r$, $p'$, $q'$, $r' \in \N_{\geq 2}$ and that $\frac{1}{p}+\frac{1}{q}+\frac{1}{r} < 1$ and $\frac{1}{p'}+\frac{1}{q'}+\frac{1}{r'} < 1$. Then the following are equivalent. 
\begin{enumerate}
    \item $\Gamma(p,q,r) \cong \Gamma(p',q',r')$.
    \item $\hut{\Gamma(p,q,r)} \cong \hut{\Gamma(p',q',r')}$.
    \item The multisets $\Menge{p,q,r}$ and $\Menge{p',q',r'}$ are equal.
\end{enumerate}
\end{thm}

The previous theorem should not be understated as it yields a bijection between the `complete Coxeter graph' for hyperbolic von Dyck triangle groups and the groups themselves, while also showing that this geometric information --- which, by Gau\ss{}--Bonnet, yields the area of the fundamental hyperbolic triangle of the given group \cite[Theorem~1.4.2]{KatokFuchsian} --- is entirely encoded by finite quotients. It also raises the question of whether a similar statement holds for all triangle groups. Making use of results of Bridson--Conder--Reid~\cite{BridsonConderReid} and some geometric arguments we shall see that, in rank at most three, finite quotients do distinguish Coxeter systems. 

\begin{thm} \label{thm:ProfRigHypRank3}
Within $\mc{G}_{\leq 3}$ all Coxeter groups are profinitely rigid. That is, if $W_1$ and $W_2$ are Coxeter groups admitting Coxeter generating sets $S_1$ and $S_2$ such that $\CCG{W_1,S_1}$ and $\CCG{W_2,S_2}$ are vertices in $\mc{G}_{\leq 3}$, then $W_1 \cong W_2$ if and only if $\hut{W_1} \cong \hut{W_2}$.
\end{thm}

\begin{proof}
The implication $W_1\cong W_2 \implies \hut{W_1} \cong \hut{W_2}$ is immediate \cite[Proposition~2.2.4]{RibesZalesskii}.

%%% both groups finite 
For the converse we first rule out a trivial case: the profinite completion of a Coxeter group is finite if and only if the given group is itself finite. We henceforth assume that the groups considered are infinite.

%%% a) both groups of rank 2
Suppose $W_1$ and $W_2$ are both infinite and $2$-generated. Then $W_1 \cong D_\infty \cong W_2$ and the claim is trivially true in this case.

%%% b) one group is a triangle group
Now suppose $W_1$ is a triangle group, say $W_1 = \Delta(p,q,r)$, and that $W_2$ has a complete Coxeter graph contained in $\mc{G}_{\leq 3}$. We want to check that $\hut{\Delta(p,q,r)} \cong \hut{W_2}$ implies $\Delta(p,q,r) \cong W_2$. 
%%% b)1) second group is not a triangle group, smaller rank case
Let us first deal with the case where $W_2$ is generated by two elements. Being infinite, $W_2$ is the infinite dihedral group $\Z \rtimes C_2$. Thus $\hut{\Delta(p,q,r)} \cong \hut{W_2} \cong \hut{\Z} \rtimes C_2$, 
where the last isomorphism follows from \cite[Proposition~2.6]{GrunewaldZalesskiiGenus}. Thus $\Delta(p,q,r)$, being a subgroup of the semidirect product $\hut{\Z} \rtimes C_2$, must be virtually abelian since its (abelian, normal) subgroup $\hut{\Z} \cap \Delta(p,q,r)$ has finite index. As is known, cf. Proposition~1 in Chapter~VI, Paragraph~2 of \cite{BourbakiLie4-6}, the triangle group $\Delta(p,q,r)$ must thus be Euclidean, so that $\{p,q,r\}$ is one of the multisets $\Menge{2,3,6}$, $\Menge{2,4,4}$, $\{2,2,\infty\}$ or $\Menge{3,3,3}$. Neither $\Menge{2,4,4}$, $\Menge{2,2,\infty}$ nor $\Menge{3,3,3}$ is possible since the corresponding triangle groups have abelianizations not isomorphic to that of $W_2 = D_\infty$. Thus $\Delta(p,q,r) = \Delta(2,3,6)$. This is also impossible as $\hut{D_\infty} \cong \hut{\Z} \rtimes C_2$ contains no copy of $\hut{\Z}^2$, whereas $\hut{\Delta(2,3,6)}$ does. It follows that an infinite triangle group $\Delta(p,q,r)$ cannot have profinite completion isomorphic to that of a Coxeter group of rank two (or less).

%%% b)2) second group is also a triangle group
Lastly, we deal with the case where $W_2$ is also a triangle group, say $W_2 = \Delta(p',q',r')$, which we assume to be infinite. Here substantially more work is required. 

%%% b)2)i) Euclidean
Assume first that $W_1 = \Delta(p,q,r)$ is Euclidean, hence (again) virtually abelian. More precisely, $\Delta(p,q,r) \cong \Z^n \rtimes W_0$ for some $n \geq 1$ and some finite Coxeter group $W_0$; see \cite[Chapter~VI]{BourbakiLie4-6}. 
Applying \cite[Proposition~2.6]{GrunewaldZalesskiiGenus} again, 
\[
\hut{\Delta(p',q',r')} \cong \hut{\Z^n \rtimes W_0} \cong \hut{\Z}^n \rtimes W_0.
\] 
Arguing as in the dihedral case, $\Delta(p',q',r')$ must itself be virtually abelian, whence $\Delta(p',q',r')$ is Euclidean. The triples $\Menge{2,3,6}$, $\Menge{2,4,4}$, $\{2,2,\infty\}$ and $\Menge{3,3,3}$ yield abelianizations $C_2^2$, $C_2^3$, $C_2^3$, and $C_2$, respectively. Since groups with isomorphic profinite completions must have the same abelianizations by~\cite[Proposition~3.1]{ReidICM}, it follows that the multisets $\Menge{p,q,r}$ and $\Menge{p',q',r'}$ either agree or one has (up to reordering) $\Menge{p,q,r}=\Menge{2,4,4}$ and $\Menge{p',q',r'}=\Menge{2,2,\infty}$. But the latter situation cannot occur, as we now explain. The former group $\Delta(2,4,4)$ is the group of type $\til{\tB}_2$ isomorphic to $\Z^2 \rtimes W_0$, where $W_0$ is the Coxeter group $D_4$ of type $\tB_2 = \tI_2(4)$. The latter group, $\Delta(2,2,\infty)$, is isomorphic to $D_\infty \times C_2 \cong (\Z\rtimes C_2) \times C_2$. Again taking profinite completions via \cite[Proposition~2.6]{GrunewaldZalesskiiGenus}, we obtain  $\hut{\Z}^2 \rtimes D_4 \cong (\hut{\Z} \rtimes C_2) \times C_2$. A contradiction, since $(\hut{\Z} \rtimes C_2) \times C_2$ contains no elements of order four, whereas $\hut{\Z}^2 \rtimes D_4$ does. 

%%% b)2)ii) hyperbolic 
Suppose now that $\Delta(p,q,r)$ is hyperbolic. 
%, i.e., $\frac{1}{p}+\frac{1}{q}+\frac{1}{r} < 1$. 
Here $W_2=\Delta(p',q',r')$ must also be hyperbolic, for otherwise it would be spherical or Euclidean, in which case its profinite completion would be finite or virtually abelian, respectively. Both possibilities would contradict the assumption that $\hut{\Delta(p,q,r)} \cong \hut{W_2}$ because $\Delta(p,q,r) \leq \hut{\Delta(p,q,r)}$ contains nonabelian free subgroups (e.g., by the Tits alternative). 

%%% b)2)ii)I) hyperbolic with infinity-labels 
Given that $\hut{\Delta(p,q,r)} \cong \hut{\Delta(p',q',r')}$ with both ${\Delta(p,q,r)}$ and ${\Delta(p',q',r')}$ being hyperbolic, we now claim that $\infty \in \{p,q,r\}$ if and only if $\infty \in \{p',q',r'\}$. To check this we need a bit more information on the structure of $\Delta(p,q,r)$. Assume, without loss of generality, that $r = \infty$. Here, $\Delta(p,q,\infty)$ splits as an amalgamated free product; cf. \cref{thm:MihalikTschantzVisual}. More precisely, reading off the visual decomposition from the underlying complete Coxeter graph~\cite{MihalikTschantzVisual}, up to reordering the indices one has
\[
\Delta(p,q,\infty) \cong  
\begin{cases}
D_p \ast_{C_2} D_q, & \text{ if } p,q < \infty, \\
D_p \ast C_2, & \text{ if } p < \infty = q, \\
C_2 \ast C_2 \ast C_2 & \text{ otherwise},
\end{cases}
\]
where $C_n$ denotes the cyclic group of order $n$ and $D_m$ denotes the dihedral group of order $2m$. Thus $\Delta(p,q,\infty)$ is virtually free; cf. \cite[Section~2.6, Proposition~11]{SerreTrees}. Since $\Delta(p,q,\infty)$ is an amalgamated free product of finite groups, it follows that 
\[
\hut{\Delta(p,q,\infty)} \cong  
\begin{cases}
D_p \coprod_{C_2} D_q, & \text{ if } p,q < \infty, \\
D_p \coprod C_2, & \text{ if } p < \infty = q, \\
C_2 \coprod C_2 \coprod C_2 & \text{ otherwise},
\end{cases}
\]
where $\coprod$ denotes the (amalgamated) free product in the category of profinite groups; see \cite[Section~4]{GrunewaldZalesskiiGenus} or \cite[Section~9.2]{RibesZalesskii}. By a result of Dyer~\cite{DyerConjSep}, the (virtually free) group $\Delta(p,q,\infty)$ is also conjugacy separable. One can thus argue exactly as in the proof of \cite[Theorem~5.1]{BridsonConderReid} to conclude that any maximal finite subgroup of $\hut{\Delta(p,q,\infty)}$ is isomorphic to a maximal finite subgroup of ${\Delta(p,q,\infty)}$. By a theorem of Tits (cf. \cite[Theorem~12.3.4(i)]{DavisBook}), these are precisely the maximal spherical parabolic subgroups: $D_p$ and $D_q$, or $D_p$ and $C_2$, or only $C_2$, whichever applicable. 

Our current assumption is that $\hut{\Delta(p,q,\infty)} \cong \hut{\Delta(p',q',r')}$ with $\{p,q,\infty\}$ and $\{p',q',r'\}$ both being hyperbolic triples, and we want to prove that at least one of $p'$, $q'$, $r'$ is also infinite. Suppose, on the contrary, that $p',q',r' < \infty$. By the previous paragraph, we have that each (maximal) spherical parabolic subgroup $D_{p'}$, $D_{q'}$, $D_{r'}$ of $\Delta(p',q',r')$ is isomorphic to a subgroup of a maximal spherical parabolic subgroup of $\Delta(p,q,\infty)$. Hence each of $p'$, $q'$, $r'$ divides some (finite) element of $\{p,q\}$ or $2$, according to the cases applicable for $\Delta(p,q,\infty)$. On the other hand, since the corresponding von Dyck subgroup $\Gamma(m,n,l) \leq \Delta(m,n,l)$ has index two, it follows from \cite[Corollary~3.3 and Proposition~3.5]{BridsonConderReid} and from the multiplicativity of $L^2$-Betti numbers (\cite[Proposition~J.5.1]{DavisBook}) that the rational Euler characteristics of $\Gamma(p,q,\infty)$ and of $\Gamma(p',q',r')$ must be equal. However, a straightforward comparison of Euler characteristics yields 
\[ \frac{1}{p'} + \frac{1}{q'} + \frac{1}{r'} = 
\begin{cases}
\frac{1}{p} + \frac{1}{q}, & \text{ if } p,q < \infty, \\
\frac{1}{p}, & \text{ if } p < \infty = q, \\
0 & \text{ otherwise},
\end{cases}
 \]
leading to contradictions in all three cases; we refer the reader, e.g., to \cite[page~319]{LiebeckShalevFuchsian} or \cite[page~664]{HughesFuchsian} for an explicit formula of the Euler characteristic of such groups in both the noncocompact and cocompact cases. (See also \cite[Section~2.9]{SerreTrees}.) In conclusion, at least one of $p'$, $q'$, $r'$ must also be infinite, as claimed. 

Thus, given $\hut{\Delta(p,q,r)} \cong \hut{\Delta(p',q',r')}$ with the triples $\{p,q,r\}$ and $\{p',q',r'\}$ being hyperbolic, there are two cases to consider. Up to reordering labels, either $r = \infty = r'$, or all of the values $p$, $q$, $r$, $p'$, $q'$, $r'$ are finite. In the former case we argue as in the previous paragraphs to deduce that $\Delta(p,q,\infty)$ and $\Delta(p',q',\infty)$ are both virtually free, in which case their maximal spherical parabolic subgroups agree with the maximal finite subgroups of the completion $\hut{\Delta(p,q,\infty)} \cong \hut{\Delta(p',q',\infty)}$. This implies that $\{D_p, D_q\} = \{D_{p'}, D_{q'}\}$ and thus, up to reordering labels, $p = p'$ and $q=q'$, yielding $\Delta(p,q,\infty) \cong \Delta(p',q',\infty)$. 

%%% b)2)ii)II) hyperbolic WITHOUT infinity-labels 
Now assume that $p$, $q$, $r$, $p'$, $q'$, $r'$ are all finite, so that $W_1 = \Delta(p,q,r)$ and $W_2 = \Delta(p',q',r')$ act cocompactly on the hyperbolic plane. Write $\phee : \hut{W_1} \to \hut{W_2}$ for the assumed isomorphism between the profinite completions. For simplicity, write $\Gamma_1 := \Gamma(p,q,r) \leq W_1$ and $\Gamma_2 := \Gamma(p',q',r') \leq W_2$. We shall prove that $\{p,q,r\} = \{p',q',r'\}$ again by using the corresponding von Dyck subgroups $\Gamma_1$ and $\Gamma_2$. Consider the subgroup $\Gamma_0 := \phee(\quer{\Gamma_1}) \cap W_2 \leq W_2$, where $\quer{X}$ denotes the (topological) closure of $X$. We argue that $\Gamma_0 \cong \Gamma_2$. Indeed, by \cite[Corollaries~2.8 and~2.9]{BridsonConderReid} one has 
\[ [W_2 : \Gamma_0] = 2 \quad \text{ and } \quad \hut{\Gamma_0} \cong \hut{\Gamma_1}.\] 
Now, by direct computational methods (e.g., \cite{ConderDobcsanyi}) or straightforward geometric arguments with hyperbolic polygons (cf. \cite[Theorem~3]{PFPHyperbolic}), one can list from the finite presentation 
\[ W_2 = \Delta(p',q',r') = \spans{\alpha,\beta,\gamma \mid \alpha^2, \beta^2, \gamma^2, (\alpha \beta)^{p'}, (\beta \gamma)^{q'}, (\gamma \alpha)^{r'}} \]
all the (possible) subgroups of index two in $W_2$. Besides the von Dyck triangle subgroup $\Gamma_2 = \Gamma(p',q',r')$, which is generated by the rotations $\alpha \beta$, $\beta \gamma$, and $\gamma \alpha$ \cite[Section~4.4]{KatokFuchsian}, the other possible subgroups of index $2$ in $W_2$ are given as follows.
\begin{align} \label{eq:listsbgps}
\begin{split}
\spans{\beta,\gamma,\alpha\beta\alpha,\alpha\gamma\alpha}, \quad \spans{\alpha,\gamma,\beta\alpha\beta,\beta\gamma\beta}, \quad \spans{\alpha,\beta,\gamma\alpha\beta,\gamma\beta\gamma}, \\
\spans{\alpha,\beta\alpha\beta,\beta\gamma}, \quad \spans{\beta,\alpha\beta\alpha,\gamma\alpha}, \, \text{ and } \, \spans{\gamma,\beta\gamma\beta,\alpha\beta};
\end{split}
\end{align}
compare \cite[Table~2]{PFPHyperbolic}. Thus $\Gamma_0$ is isomorphic to $\Gamma_2$ or to one of the groups from the list~\eqref{eq:listsbgps}. By \cite[Theorem~2.4.1]{KatokFuchsian}, the finite subgroups of (the Fuchsian) group $\Gamma_1 = \Gamma(p,q,r)$ are all cyclic. Applying \cite[Theorem~5.1]{BridsonConderReid} to $\Gamma_1$, the same holds for its completion $\hut{\Gamma_1}$. Since $\hut{\Gamma_1} \cong \hut{\Gamma_0} \supset \Gamma_0$, any finite subgroup of $\Gamma_0$ must also be cyclic. However, as can be seen by direct inspection, the (possible) groups contained in the list~\eqref{eq:listsbgps} all contain noncyclic finite subgroups (namely, dihedral groups). Thus, the only possibility for $\Gamma_0$ is to coincide with $\Gamma_2$, as desired. This being the case, we have that $\hut{\Gamma_1} \cong \hut{\Gamma_0} \cong \hut{\Gamma_2}$. Here, we can apply \cref{thm:BCR} to conclude that $\{p,q,r\} = \{p',q',r'\}$ and we are done.
\end{proof}

We stress that our \cref{thm:ProfRigHypRank3} also takes into account triangle groups with cusps --- i.e., those whose Dirichlet fundamental domain contains vertices at infinity. In contrast, Bridson--Conder--Reid work with cocompact von Dyck triangle groups in \cite[Section~8]{BridsonConderReid}. The following result gives a direct analogue of their \cref{thm:BCR} in the Coxeter setting while also allowing for cusps.

\begin{thm} \label{thm:analogueBCR}
Let $\Menge{p,q,r}$ and $\Menge{p',q',r'}$ be multisets of elements in $\N_{\geq 2} \cup \Menge{\infty}$. Then the following are equivalent. 
\begin{enumerate}
    \item $\Delta(p,q,r) \cong \Delta(p',q',r')$.
    \item $\hut{\Delta(p,q,r)} \cong \hut{\Delta(p',q',r')}$.
    \item $\Menge{p,q,r} = \Menge{p',q',r'}$.
\end{enumerate}
\end{thm}

\begin{proof}
The implications (iii) $\implies$ (i) $\implies$ (ii) are immediate. The claim (ii) $\implies$ (i) is a special case of \cref{thm:ProfRigHypRank3}. So we need only prove that (i) implies (iii).

For the finite triangle groups (of which there are only finitely many), the equality $\Menge{p,q,r} = \Menge{p',q',r'}$ follows by direct inspection of their Coxeter--Dynkin diagrams after Coxeter's classification theorem. If $\Delta(p,q,r)$ (or $\Delta(p',q',r')$) act geometrically on $\R^2$ or $\Hyp^2$, a theorem of Charney--Davis \cite{CharneyDavis} shows that $\Delta(p,q,r)$ (or $\Delta(p',q',r')$) is an isolated point in $\mc{G}$. Hence $\Menge{p,q,r} = \Menge{p',q',r'}$. The result of Charney--Davis is applicable to all (infinite) triangle groups without an $\infty$-labeled edge. 

We may henceforth assume that $r = \infty = r'$. Since the only Euclidean triple containing $\infty$ is $\Menge{2,2,\infty}$, we may further assume that the given triangle groups are hyperbolic. As in the proof of \cref{thm:ProfRigHypRank3} we may resort to an Euler characteristic argument to clear the remaining cases. Indeed, the assumed isomorphism $\Delta(p,q,\infty) \cong \Delta(p',q',\infty)$ yields 
\[\chi_\Q(\Gamma(p,q,\infty)) = \frac{1}{2}\cdot\chi_\Q(\Delta(p,q,\infty)) = \frac{1}{2}\cdot\chi_\Q(\Delta(p',q',\infty)) = \chi_\Q(\Gamma(p',q',\infty)).\]
Hence (cf. \cite[page~319]{LiebeckShalevFuchsian} or \cite[page~664]{HughesFuchsian} for a formula) we obtain 
\begin{equation} \label{eq:Euler426}
\frac{1}{p} + \frac{1}{q} = \frac{1}{p'} + \frac{1}{q'} < 1,
\end{equation}
where $\frac{1}{\infty}$ is read off as $0$. Moreover, using Tits' theorem \cite[Theorem~12.3.4(i)]{DavisBook} to compare the (possible) maximal spherical parabolic subgroups of $\Delta(p,q,\infty)$ and $\Delta(p',q',\infty)$ --- which are of order $2$ or dihedral of orders in $\{2p,2q,2p',2q'\}$, depending on the given labels --- we also have the requirement 
\begin{align} \label{eq:sphericalparabolics426}
\begin{split}
\infty > x \in \{p,q\} \implies \exists x' \in \{2,p',q'\} \text{ such that } x \text{ divides } x', \quad \text{ and} \\
\infty > x' \in \{p',q'\} \implies \exists x \in \{2,p,q\} \text{ such that } x' \text{ divides } x.
\end{split}
\end{align}
Exhausting all possibilities in Condition~\eqref{eq:sphericalparabolics426} and substituting values in \cref{eq:Euler426}, one readily checks that $\{p,q\} = \{p',q'\}$, which finishes off the proof.
\end{proof}

\begin{cor}[{A complete picture of $\mc{G}_{\leq 3}$}] \label{cor:Rank3is1D}
The layers $\mc{G}_{\leq 3}$ form a disconnected $1$-dimensional subcomplex of the Coxeter galaxy and of its vertical core. More precisely, $\mc{G}_{\leq 3} = \VC_{\leq 3}$ and the following hold.
 \begin{enumerate}
     \item \label{Rank3case1} The vertices from $\mc{G}_{\leq 2}$ are isolated in $\mc{G}_{\leq 3}$, except for the dihedral groups $\tI_2(4k+2)$ with $k \geq 1$. Each such dihedral group has exactly one neighbor, namely the (spherical) triangle group $\Delta(2k+1,2,2)$. 
     \item All vertices in $\mc{G}_{\leq 3}$ distinct from the above exceptions are also isolated.
 \end{enumerate}
\end{cor}

\begin{proof}
The edges arise from pseudo-transpositions applied to certain dihedral groups in $\mc{G}_2$. The relevant groups are, by~\cref{lem:blowingup}, exactly those listed in~\eqref{Rank3case1}.  
Apply \cref{thm:ProfRigHypRank3} to conclude that $D_\infty$ is also an isolated point. The remaining cases of (infinite) triangle groups follow immediately from \cref{thm:analogueBCR}.
\end{proof}

\begin{cor}\label{cor:SolutionIsoProbRk3}
The isomorphism problem is decidable for groups in the layers ${\mc{G}_{\leq 3}}$.
\end{cor}

\begin{proof}
Straightforward from \cref{cor:Rank3is1D} as the connected components of vertices in $\mc{G}_{\leq 3}$ are completely determined by the multiset of edge labels of their underlying complete Coxeter graphs.
\end{proof}

To the best of our knowledge, \cref{cor:SolutionIsoProbRk3} has not explicitly appeared in the literature before, though it was pointed out to us by Piotr~Przytycki that it can be deduced from the works of Caprace--M\"uhlherr~\cite{CapraceMuehlherr2-Spherical} and Weigel~\cite{WeigelTwistTriangle}. Alternatively, it also follows from Ratcliffe--Tschantz's solution to the twist conjecture for chordal Coxeter groups \cite{RatcliffeTschantz}. We remark, however, that both such solutions have an intermediate step shown in an unpublished preprint \cite{HowlettMuehlherr} by Howlett--M\"uhlherr; see also \cite[Appendix~B]{CapracePrzytyckiTwistRigid}.

%%%%%%%%%%%%%%%%%%%%%%%%%%%%%%%%%%%%%%%%%%%%%%%%%%
% Part 3: Survey of known results and conjectures
%%%%%%%%%%%%%%%%%%%%%%%%%%%%%%%%%%%%%%%%%%%%%%%%%%

\section{History of the problem, classical results and conjectures}
\label{sec:History}

In several subclasses of Coxeter groups the isomorphism problem is known to be decidable. And there exist ways to reduce the full problem to a subproblem about saturated, angle-compatible Coxeter systems. A key tool in these reductions are certain types of horizontal moves (edges) within the Coxeter galaxy. 
We describe these horizontal moves in \cref{subsec:horizontalmoves}. In \cref{subsec:twistConj} we explain how the twist conjecture provides a way to solve the isomorphism problem. Finally we summarize other approaches and known solutions to the isomorphism problem within certain subclasses of Coxeter groups in \cref{subsec:survey}.

%%%%%%%%%%%%%%%%%%%%%%%%%%%%%%%%%%%%%%%%%%%%%%%%%%%%
%%%%%%%%%%%%%%%%%%%%%%%%%%%%%%%%%%%%%%%%%%%%%%%%%%%%

\subsection{Horizontally navigating within a layer}
\label{subsec:horizontalmoves}

In this section we describe the known types of horizontal moves: twists, cross-eyed twists and $J$-deformations.

%%%%%%%%%%%%%%%%%%%%%%%%%%%%%%%%%%%%%%%%%%%%%%%%%%%%
\subsubsection{Twisting the diagram}\label{subsec:twists}
Twists were discovered by M\"uhlherr by means of an example in \cite{BernhardExample}, and later formally introduced in \cite{BMMN}. The definition below is essentially the one provided in \cite{CapracePrzytyckiTwistRigid}. Compare also \cref{rem:trivialTwist}. 
Given a subset $J\subseteq S$ of a generating set $S$ of a Coxeter system, recall that $J^\perp$ consists of those elements in $S\setminus J$ that commute with all elements in $J$.

\begin{dfn}[Diagram twist]
\label{def:twist}
Let $(W,S)$ be a Coxeter system and $\CCGo$ its complete Coxeter graph. 
Then $\CCGo$ \emph{admits a twist} if there exist subsets $A,B,J$ of $S$ as follows.  
The set $J\subseteq S$ generates an irreducible, spherical Coxeter group. Denote by $w_J$ the longest element in that group. The subsets $A$, $B \subseteq S$ are disjoint and such that 
\begin{enumerate}
    \item $S\setminus(J\cup J^\perp) = A\cup B$, and
    \item all edges between $A$ and $B$ have label $\infty$. 
\end{enumerate}
We then define the \emph{twist of $B$ along $J$} to be the map $\tau: S\to W$ where 
\begin{equation*}
\tau(s)=\begin{cases}
          s \quad &\text{if } \, s \in A, \\
          w_Jsw_J \quad &\text{otherwise}.
     \end{cases}
\end{equation*}
The image $\tau(S)$  is called \emph{twisted generating set} of $W$. 
The \emph{twisted complete Coxeter graph} $\CCGo'$ is the complete edge-labeled graph with vertices $\tau(S)$ and, for all $s\neq t$ in $S$, the edge in $\CCGo'$ between $\tau(s)$ and $\tau(t)$  has label $\ord{\tau(s),\tau(t)}=\mathrm{ord}(\tau(s)\tau(t))$.
\end{dfn}

A twist depends both on the choice of $J$ and on the choice of $A$ and $B$. 
For a given $J$ there may be several choices for $A$ and $B$ in $S\setminus (J\cup J^\perp)$. Every twist is invertible in the sense that we get back the original Coxeter generating set and complete diagram when we twist $B^{w_J}$ along the same set $J$.  

It was shown in \cite[Theorem 4.5]{BMMN} that the pair $(W,S')$, with $S'$ a twisted generating set as above in \cref{def:twist}, is also a Coxeter generating set of $W$. Hence the twisted Coxeter generating set and twisted complete Coxeter graph deserve their names.

\begin{remark}[Discussion of the definition of twists]
\label{rem:trivialTwist}
In the definition of twist provided in \cite{CapracePrzytyckiTwistRigid} (called `elementary twist' there), the possibility that $A\cap B$ is nonempty or that either $A$ or $B$ is empty was not excluded. 

However, the fact that we assume $A\cup B$ to be a disjoint union in \cref{def:twist} is not an important restriction, and is also done, for instance, in M\"uhlherr's definition of twist in the survey~\cite{BernhardSurvey} and in \cite[p.~2083]{HuangPrzytycki}. Suppose that $A\cap B\neq \leer$. In this case the map $\tau$ is the identity on all of $A$ (including those elements that are also in $B$) and equals conjugation by $w_J$ on $B\setminus A$. We may hence replace $B$ by $B\setminus A$ and obtain a disjoint union without changing the map $\tau$.

The definition in \cite{BMMN} coincides with the definition in \cite{CapracePrzytyckiTwistRigid} if one chooses $V=J$ and $U=B$. Similarly, the definition in~\cite{BernhardSurvey} coincides with \cref{def:twist} by choosing $K = A$ (and thus $L = B$).
\end{remark}

In order to facilitate drawing the twisted graph for $W$ with respect to the twisted Coxeter system we will now explicitly state how to obtain $\CCGo'$ from $\CCGo$. This might also explain why we call it a \emph{twist of $B$}.  

\begin{remark}[Drawing the twisted Coxeter graph]
Let $(W,S)$, the graph $\CCGo$, a subset $J\subseteq S$ and subsets $A,B \subseteq S$ be given as in \cref{def:twist}. 
Recall that the set $S$ of vertices in $\CCGo$ is the following disjoint union: 
\[
S= A\cup B\cup J\cup J^\perp.  
\]
Check with \cref{def:twist} that $\tau(A\cup J\cup J^\perp)=A\cup J\cup J^\perp$ as $\tau$ restricted to $A$ or $J^\perp$ is the identity and permutes $J$. 
Thus set-wise we have the equality 
\[
\tau(S)= A\cup \tau(B)\cup \tau(J)\cup J^\perp = A \cup B^{w_J} \cup J\cup J^\perp,  
\]
where $B^{w_J}$ denotes conjugation of $B$ by $w_J$. The graph $\CCGo'$ is then obtained as follows. 
Remove from $\CCGo$ all vertices $b \in B$ and with them all edges connecting $b$ to other elements. 
Then add for each $b\in B$ a new vertex $\tau(b)=b^{w_J}$. Also add, for each $b\in B$ and $s\in \tau(S)$, an edge between $\tau(b)=b^{w_J}$ and $s$ with label 
\begin{equation*}
\ord{\tau(b),s}=\begin{cases}
          \ord{b,b'} \quad &\text{if } \, s=\tau(b') \text{ for some } b'\in B, \\
          \infty    \quad &\text{if } \, s=\tau(a)=a \text{ for some } a\in A, \\
          \ord{b,j}  \quad &\text{if } \, s=\tau(j) \text{ for some } j\in J,\\
          \ord{b,i}  \quad &\text{if } \, s=i=\tau(i) \text{ for some } i\in J^\perp.
     \end{cases}
\end{equation*}
Short computations yield that this agrees with \cref{def:twist}. 
The graph $\CCGo'$ hence differs from $\CCGo$ only in the vertices $B$ and edges connecting $B^{w_J}$ to permuted vertices within $J$.
\end{remark}

The obvious question at hand is whether a twisted diagram $\CCGo'$ always differs from $\CCGo$. The answer is no. If, for example, the element $w_J$ is central in $W_J$, then $\CCGo'$ and $\CCGo$ are isomorphic.  This is in particular the case if $\vert J\vert=2$ and the corresponding parabolic subgroup is of classical type $\tI_2(2k)$ for some $k\geq 1$. 
Moreover, if $J$ contains only a single element the graphs $\CCGo'$ and $\CCGo$ are canonically isomorphic. This motivates the following definition. The term nontrivial twist was defined implicitly in Remark 4.7 of \cite{BMMN}. 

\begin{dfn}[Nontrivial twists]
\label{def:nte twist}
A twist for $\CCGo$ is \emph{nontrivial} if $\CCGo$ and the twisted diagram $\CCGo'$ are not isomorphic as edge-labeled graphs. It is \emph{trivial} otherwise. 
\end{dfn}

We now examine which twists are nontrivial. 

\begin{lem}
Let $(W,S)$ be a Coxeter system with complete Coxeter graph $\CCGo$. Let $\tau$ be a twist of $B$ along $J$ as in \cref{def:twist}. Then $\tau$ is trivial if one of the following holds: 
\begin{enumerate}
    \item $w_J$ is central in $W_J=\langle J\rangle$,
    \item $A=\leer$, or 
    \item $B=\leer$.
\end{enumerate}
\end{lem}
\begin{proof}
If $w_J$ is central in $W_J$ the permutation on $J$ induced by $\tau$ is trivial. All connecting edges from $b\in B$ to $s\in J$ are replaced by edges connecting $b^{w_J}$ to $s$ that carry the same label. Hence the complete Coxeter graphs of $S$ and $\tau(S)$ are canonically isomorphic. 

Suppose that $A=\leer$. In this case the definition above results in $\tau$ being the conjugation of $S$ by $w_J$. Such a map does not change the underlying complete Coxeter graph. 

If $B=\leer$, then $\tau$ is the identity on $S$. Therefore the complete Coxeter graph remains unchanged. 
\end{proof}

Note that $w_J$ is trivially central in $W_J$ if $\vert J\vert=1$. 

As stated in \cite{BMMN}, the finite Coxeter groups in which the element $w_J$ in \cref{def:twist} is not central are those of type $\tA_n$ for $n\geq 2$, $ \tD_n$ for $n$ odd, $\tE_6$ or $\tI_2(m)$ for odd $m$; see \cite[{\S}~7]{BrieskornSaito} for details. These are the only cases in which twists may be nontrivial.

To get a nontrivial twist we have to make sure that we are not in the situation that the map $\tau$ defined by partial conjugation with $w_J$ as in \cref{def:twist} results in a trivial move. In the last two cases, where either $A$ or $B$ were equal to the empty set, the map $\tau$ had the property that it was a (potentially trivial) permutation of $S$. We have to exclude that this happens when both  $A\neq\leer$ and $B\neq \leer$.

\begin{lem}
\label{lem:reallyTrivialTwist}
Suppose $(W,S)$ is a Coxeter system and let $J$, $A$, $B \subseteq S$ be as in \cref{def:twist}. Suppose further that the resulting twist $\tau$ of $B$ along $J$ satisfies $\tau(S)=S$. Then $A=\leer$ or $B=\leer$.
\end{lem}
\begin{proof}
So suppose that $J$, $A$ and $B$ are as in \cref{def:twist}. Observe that $\tau\vert_B$ is a permutation of $B$ as a consequence of $\tau(S)=S$. 
As $J^\perp \cap B=\leer$ such a permutation will not have a fixed point. 
Let $b\in B$. Then $\tau(b)=w_J b w_J$ is a reflection in the Coxeter group $W$. Since $B \cap J = \leer$ and $w_J$ is the longest element of the standard parabolic subgroup $\spans{J} \leq W$, a reduced decomposition for the reflection $\tau(b)$ (with respect to the original Coxeter generating set $S$) involves all letters in $J \cup \{b\}$. Thus the smallest parabolic subgroup containing $\tau(b)$ is spanned by $J$ and $b$, and strictly contains $\langle J\rangle$; cf. \cite[Chapter~IV, Paragraph~1, Section~8]{BourbakiLie4-6}. Therefore $b$ is mapped to a reflection not contained in the parabolic subgroup spanned by $B$, and in particular is not mapped to some $b'\in B\setminus \Menge{b}$. This yields a contradiction. 
\end{proof}

Conjecturally, twists are the only relevant horizontal moves when it comes to solving the isomorphism problem (more on that in \cref{subsec:ReductionIsoProblem}). In \cref{subsec:twistConj} we formally discuss the twist conjecture which, in essence, states just that.

%%%%%%%%%%%%%%%%%%%%%%%%%%%%%%%%%%%%%%%%%%%%%%%%%%%%
\subsubsection{Angle-deformations, cross-eyed twists  and other moves}

Before we proceed to explain the various reductions of the isomorphism problem and how they relate to the twist conjecture let us comment on other existing horizontal moves. 
We start with angle-deformations, introduced by Marquis and M\"uhlherr in Definition~3.1 of \cite{MarquisMuehlherr}. 

\begin{dfn}(Angle-deformations)
	\label{def:angle deformation}
	Let $(W,S)$ be a Coxeter system and $J=\{r,s\}\subseteq S$ an edge in its Coxeter--Dynkin diagram $\Gamma = \Gamma(W,S)$. Choose $w\in\langle J\rangle$ such that $\langle wrw^{-1}, s\rangle=\langle J\rangle$. 
	An \emph{$(r,s,w)$-deformation} of $S$ is a map $\delta: S\to W$ such that the following hold: 
	\begin{itemize}
		\item[(AD1)] $\delta(S)\subseteq S^W$,
		\item[(AD2)] $\delta(s)=s$ and $\delta(r)=wrw^{-1}$, 
		\item[(AD3)] $\delta(S)$ is a Coxeter generating set of $W$, and  
		\item[(AD4)] there exists a bijection $\Delta$ from the set of edges of $\Gamma$ to the set of edges of $\Gamma':=\Gamma(W, \delta(S))$  such that $\Delta(J)=\{wrw^{-1}, s\}$ and such that for all edges $K\neq J$ in $\Gamma$ there exists some $w_K\in W$ with $\Delta(K)=K^{w_K}$. 
	\end{itemize}
	A \emph{$J$-deformation} is an $(r,s,w)$-deformation for $J=\{r,s\}$ and some choice of $w$. By an \emph{angle-deformation} we mean a $J$-deformation for some $J$. 
\end{dfn}

Some but not all angle-deformations extend to automorphisms of $W$. Properties like this are discussed for example in Proposition 3.4 of \cite{MarquisMuehlherr}. 
Marquis and M\"uhlherr also provide characterizations of configurations which allow for angle-deformations.  

\subsubsection{Cross-eyed twists}
The definition of a cross-eyed twist is similar to, but technically more involved than the definition of an angle-deformation.  
Cross-eyed twists were introduced by Ratcliffe and Tschantz \cite[p.~69]{RatcliffeTschantz} for chordal Coxeter groups only. It is worth mentioning that this paper motivated the work of Marquis and M\"uhlherr in \cite{MarquisMuehlherr}. 

The name `cross-eyed twist' probably comes from the situation illustrated in Figure 3 of \cite{RatcliffeTschantz} where in a diagram looking like a crossed-out eye certain edges are swapped by a cross-eyed twist. They are also referred to as \emph{$5$-edge angle-deformations} and are indeed a deformation of the Coxeter--Dynkin diagram obtained from certain subconfigurations of the diagram in relation to an edge $J=\{r,s\}$ with label $5$. 
Certain (good irreducible) subsets $A$ of $\Gamma$ are twisted by a conjugation with a suitably chosen element $w_A$ in $W$. Other (bad irreducible simplices) $A\subseteq S$ are replaced by $w_A\beta_A(A)w_A^{-1}$ for some well chosen $w_A$ and a non-inner automorphism $\beta_A: \langle A\rangle \to \langle A\rangle$.  

We will not state the precise definition here but refer the reader to \cite[Theorem~5.2]{RatcliffeTschantz} and the definition right below it. It would be interesting to know precisely how cross-eyed twists and angle-deformations are related. Both moves were used to solve or reduce the saturated isomorphism problem. See \cref{subsec:ReductionIsoProblem} for details. 

\subsubsection{Further comments on horizontal moves}

There is also a notion of an \emph{$S$-transvection}. Such a move is another example of a non-reflection preserving modification of a given Coxeter generating set $S$. It was introduced by Howlett and M\"uhlherr in an unpublished preprint \cite{HowlettMuehlherr}. See Section~4 of \cite{BernhardSurvey} for a (partial) definition. 
As M\"uhlherr states, this is not an isomorphism easily seen from the diagram. 

Another type of move also mentioned in \cite{BernhardSurvey} is the following. 
We call a subset $J\subseteq S$ of a Coxeter generating set $S$ of $W$ a \emph{graph factor} of $(W, S)$ if $J$ is spherical and if for all $t\in S\setminus J$ either $t \in J^\perp$ or the order of $tj$ is infinite for all $j\in J$. If $J$ is a graph factor and $\alpha$ an automorphism of $\langle J\rangle$, then one can verify that there is a unique automorphism of $W$ stabilizing the subgroup $\langle J\rangle$, inducing $\alpha$ on it and inducing the identity on $S\setminus J$. Such an automorphism is then called \emph{$J$-local}. 
The $S$-transvections and $J$-local automorphisms are used to reduce the isomorphism problem; see \cite[Section~4]{BernhardSurvey} for details.

%%%%%%%%%%%%%%%%%%%%%%%%%%%%%%%%%%%%%%%%%%%%%%%%%%%%
%%%%%%%%%%%%%%%%%%%%%%%%%%%%%%%%%%%%%%%%%%%%%%%%%%%%
\subsection{The twist conjecture}
\label{subsec:twistConj}

As discussed in the previous subsection, twists provide a way to move within a fixed layer of the Coxeter galaxy. For a given Coxeter system $(W,S)$ with complete Coxeter graph $\CCGo$ one may obtain by a nontrivial twist a nonisomorphic complete Coxeter graph $\CCGo'$ which defines the same group. 
It is shown in \cite{BMMN} that not only does $\CCGo$ have the same rank as $\CCGo'$, they also define the same set of reflections. 
This theorem motivated the following conjecture, which was first stated as Conjecture~8.1 in \cite{BMMN}. See also \cite[Conjecture~7.3]{NuidaSurvey}. 
Note that any two reflection-compatible Coxeter systems for a same $W$ automatically have the same rank by Theorem~3.8. in \cite{BMMN}. 

\begin{conj}[The (reflection-compatible) twist conjecture]
\label{conj:rc-twistConjecture}
Let $(W,S)$ and $(W,S')$ be two reflection-compatible (cf. \cref{def:reflectioncompatible}) Coxeter systems for the same group $W$. Let $\CCGo$ and $\CCGo'$ be their respective complete Coxeter graphs. Then the complete graph $\CCGo'$ can be obtained from $\CCGo$ by a finite sequence of twists.  
\end{conj}

In other words, it was conjectured that Coxeter systems are \emph{reflection-rigid} up to diagram twisting. See for example \cite{NuidaSurvey} for a precise definition. 

\begin{dfn}[Twist equivalence]
Suppose $\CCGo$ and $\CCGo'$ are two complete Coxeter graphs. We say that $\CCGo$ and $\CCGo'$ are \emph{twist equivalent} if one can be obtained from the other by a finite sequence of twists. 
\end{dfn}

Twist equivalent graphs are thus contained in a same layer of a connected component of the Coxeter galaxy. 

A natural question to ask is whether the intersection of a connected component with a fixed layer $\mc{G}_k$ can be computed by twists only. In other words, are twists enough to obtain all Coxeter graphs of a fixed rank for a given Coxeter group?  
Stated like this, the answer is easily seen to be `no' as not all vertices in a layer of a connected component are reflection-compatible.  

\begin{exm}\label{ex:twist-not-enough}
Consider the starlet $\CCGo$ on three vertices and edge labels $4k+2$, $4l+2$ and $\infty$, with $k \neq l$. This starlet has two distinct pseudo-transpositions, along which we can blow it up following \cref{lem:blowingup}. Let $\CCGo_1$ (resp. $\CCGo_2$) be obtained from $\CCGo$ by blowing up the edge with label $4k+2$ (resp. $4l+2$). For an illustration in case $k=1$ and $l=2$ see \cref{fig:example_verticalcore} with the vertices labeled by the graphs shown in \cref{fig:isomorphicGroups}. 

By construction and by \cref{def:Galaxy}, there must be an (horizontal) edge connecting $\CCGo_1$ to $\CCGo_2$ within their layer $\mc{G}_4$. This edge, however, cannot arise from a diagram twist. Indeed, let 
\[ W \cong \spans{x,c,b,w \mid x^2, c^2, b^2, w^2, (xc)^{2k+1}, (cb)^{4l+2}, (xw)^2, (cw)^2} \]
be the Coxeter presentation for the group corresponding to $\CCGo_1$. In the notation of \cref{def:twist}, one can only twist $\CCGo_1$ along $J = \Menge{x,c}$ or along $J= \Menge{c,b}$. 

The first case, $J = \Menge{x,c}$, gives $J^\perp = \Menge{w}$ and forces $A = \leer$ (otherwise this twist is already trivial). Applying the twist along this choice of $J$ yields the new Coxeter generating set $\Menge{x,c,w,w_J b w_J}$. Since $\mathrm{ord}(x w_J b w_J) = 4l+2$ and $\mathrm{ord}(w w_J b w_J) = \infty = \mathrm{ord}(c w_J b w_J)$, it follows that $\CCGo_1$ and the twisted complete Coxeter graph are isomorphic (as edge-labeled graphs) by mapping $x \mapsto c$, $c \mapsto x$, $w \mapsto w$ and $b \mapsto w_J b w_J$. In other words, the given twist is trivial.

In the second case, $J = \Menge{c,b}$, one has $J^\perp = \leer$. Again, in order to get a possibly nontrivial twist, this choice of $J$ forces $A = \leer$. Thus $B = \Menge{x,w}$. Since $4l+2$ is even, the longest element $w_J$ is central in $\spans{J}$, so that the twisted Coxeter generating set $\Menge{w_J x w_J, c, b, w_J w w_J}$ is merely a conjugate of $\Menge{x,c,b,w}$. Thus the given twist is trivial also in this case.
\end{exm}

It hence becomes clear that one cannot expect that twists provide a full picture of every layer of the galaxy. This example, however, does not imply that \cref{conj:rc-twistConjecture} is false. 
An actual counterexample to \cref{conj:rc-twistConjecture} was provided in Figure 3 of \cite{RatcliffeTschantz} by Ratcliffe and Tschantz. See also Section 7 of \cite{NuidaSurvey} for a related discussion. The example of Ratcliffe--Tschantz provides two nonisomorphic but reflection-compatible graphs which are not twist equivalent but define the same group.

One can see that the twist conjecture needs to be restricted to the top layer, that is, to saturated Coxeter systems and graphs.

\begin{qst}[Reachability Question]
	\label{qst:reachability}
Does there exist a finite list of vertical and horizontal moves such that, for any vertex in the Coxeter galaxy, any other vertex in its connected component can be reached by applying a finite sequence of these moves? 
\end{qst}

The counterexample to \cref{conj:rc-twistConjecture} by Ratcliffe and Tschantz suggests to replace reflection-compatible systems by angle-compatible ones. 

\begin{dfn}[Angle-compatible]
\label{dfn:angle compatible}
	Two Coxeter generating sets $S$ and $S'$ of a Coxeter group $W$ are called \emph{angle-compatible} if for all $s,t\in S$ with $\langle s,t \rangle$ finite the set $\{s,t\}$ is conjugate to some pair $\{s',t'\}$ in $S'$. 
    We also refer to the Coxeter systems or any of their graphs or diagrams as being \emph{angle-compatible}. 
\end{dfn} 

By taking $s=s'$ one can see that angle-compatible systems are  reflection-compatible.

\begin{conj}[The (saturated) angle-compatible twist conjecture]
\label{conj:ac-twistConjecture}
Let $(W,S)$ and $(W,S')$ be two (saturated) angle-compatible Coxeter systems. Let $\CCGo$ and $\CCGo'$ be their respective complete Coxeter graphs. Then the complete graph $\CCGo'$ can be obtained from $\CCGo$ by a finite sequence of twists.  
\end{conj}

\cref{conj:ac-twistConjecture} (in the nonsaturated version) was stated as Conjecture 2 in \cite{BernhardSurvey}. As we will explain in \cref{subsec:ReductionIsoProblem} its proof might lead to the existence of a solution to the isomorphism problem.

Instances for which the angle-compatible twist conjecture is solved are, for example, the following: M\"uhlherr and Weidmann established the conjecture for large-type Coxeter groups \cite{MuehlherrWeidmann}; Ratcliffe and Tschantz showed that chordal Coxeter groups satisfy the conjecture \cite{RatcliffeTschantz}; Weigel proved the conjecture for Coxeter groups of rank at least three and not containing parabolic subgroups of certain types \cite{WeigelTwistTriangle}; 
And more recently Huang and Przytycki \cite{HuangPrzytycki} proved the conjecture for certain Coxeter groups of FC-type. 

\begin{remark}[{Adjustments on the twist conjecture}]
{Some more clarification about the twist conjecture and its variants is due. Although we formulated them here for complete Coxeter graphs, the original conjectures were stated for Coxeter systems. While the systems and their diagrams are usually interchangeable, there is a subtlety around the twist conjecture that might require them to be looked separately. 
Despite this subtlety being clear to experts, we explain it in the sequel for convenience of the reader.}

{As originally stated, the (angle-compatible, nonexpanded) twist conjecture \cite[Conjecture~2]{BernhardSurvey} simply predicts that angle-compatible Coxeter generating sets for the same Coxeter group can be transformed into one another by a (finite) sequence of (possibly trivial) twists. On the other hand, the conjecture typically cited in the literature as M\"uhlherr's twist conjecture is often meant as follows: if $S,S' \subseteq W$ are angle-compatible Coxeter generating sets for the same Coxeter group $W$, then there exists a finite sequence of twists and/or conjugations sending $S$ to $S'$; see, for instance, \cite[penultimate paragraph of p.~2244]{CapracePrzytyckiTwistRigid} or \cite[second paragraph of p.~2083]{HuangPrzytycki}.}

{The above `popular correction' of the twist conjecture amounts to adapting the definition of {twist equivalence} \cite{BernhardSurvey} as to allow for conjugations. That is, $S$ and $S'$ are called \emph{twist equivalent} if they differ by conjugations in $W$ and (finitely many) twists. Without this modification, the twist conjecture as originally stated (\cite[Conjecture~2]{BernhardSurvey}) cannot hold true because not every conjugation is realized by a twist. Alternatively, if the twist conjecture only allows for nontrivial twists (in particular, does not take arbitrary conjugations into account), then the validity of \cite[Conjecture~2]{BernhardSurvey} would contradict \cite[Theorem~1.1]{CapracePrzytyckiTwistRigid}.}

{Altogether, M\"uhlherr's twist conjecture says that, if $S,S' \subseteq W$ are angle-compatible, then there exists a (possibly trivial) $w \in W$ such that $S'$ is obtained from the conjugate $S^w$ by applying finitely many twists. As seen above, this is consistent with the literature \cite{CapracePrzytyckiTwistRigid,RatcliffeTschantz,WeigelTwistTriangle,HuangPrzytycki} but also with the adjusted version of {twist rigidity}: a system $(W,S)$ is called \emph{twist rigid} if it does not admit a {nontrivial} twist that yields a conjugate system; cf. \cite[p.~6]{BernhardSurvey}.} 

{Lastly, we stress that M\"uhlherr's twist conjecture (with or without conjugations) and \cref{conj:ac-twistConjecture} are compatible in our language by \cref{rem:invisibleinfo}(ii). Indeed, if two systems $(W,S)$ and $(W,S')$ differ by a conjugation and a sequence of twists, then their respective complete Coxeter graphs $\CCGo$ and $\CCGo'$ can be obtained from one another by a sequence of twists.}
\end{remark}

%%%%%%%%%%%%%%%%%%%%%%%%%%%%%%%%%%%%%%%%%%%%%%%%%%%%
%%%%%%%%%%%%%%%%%%%%%%%%%%%%%%%%%%%%%%%%%%%%%%%%%%%%

\subsection{Reduction of the isomorphism problem}
\label{subsec:ReductionIsoProblem}

This section contains an overview on how the isomorphism problem reduces to certain subproblems. 
We summarize all known reductions in \cref{fig:reduction} and will explain how the various arrows in this figure are obtained. Implications from top to bottom and left to right in the two columns of this figure simply follow from the definitions and by restriction of the problem.   
\smallskip

\begin{figure}[htb]
    \begin{center}
\begin{tikzpicture}[node distance=50pt]
\node (b1) [fill=gray!20,draw,rounded corners,align=center,anchor=center] 
      {(IP) \\ Isomorphism Problem};
\node (ph1) [right=of b1,anchor=center,align=center]
      {\phantom{(ACTC) Angle-Compatible} \\ \phantom{Twist Conjecture}};
\node (b2) [right=of ph1,fill=gray!20,draw,rounded corners,align=center,anchor=center]         
      {(SIP) \\ Saturated \\ Isomorphism Problem};
\node (b4) [below=of b2,fill=gray!20,draw,rounded corners,align=center,anchor=center] 
      {(SRIP) \\ Saturated and \\ Reflection-Compatible \\ Isomorphism Problem};
\node (ph2) [left=of b4,anchor=center,align=center] 
      {\phantom{(ACTC) Angle-Compatible} \\ \phantom{Twist Conjecture}};
\node (b3) [left=of ph2,fill=gray!20,draw,rounded corners,align=center,anchor=center] 
      {(RIP) \\ Reflection-Compatible \\ Isomorphism Problem};
\node (b6) [below=of b4,fill=gray!20,draw,rounded corners,align=center,anchor=center]  
     {(SAIP) \\ Saturated and \\ Angle-Compatible \\ Isomorphism Problem};
\node (ph3) [left=of b6,anchor=center,align=center] 
      {\phantom{(ACTC) Angle-Compatible} \\ \phantom{Twist Conjecture}};
\node (b5) [left=of ph3,fill=gray!20,draw,rounded corners,align=center,anchor=center] 
      {(AIP) \\ Angle-Compatible \\ Isomorphism Problem};
\node (ph4) [below=of b6,anchor=center,align=center] 
      {\phantom{(ACTC) Angle-Compatible} \\ \phantom{Isomorphism Problem}};
\node (b7) [left=of ph4,fill=gray!20,draw,rounded corners,align=center,anchor=center]
      {(ATC) \\ Angle-Compatible \\ Twist Conjecture};
% top to botom  - or -  bottom to top
\draw[double equal sign distance,arrows=-Implies,thick] (b1) -- (b3);
\draw[double equal sign distance,arrows=Implies-Implies,thick] (b3) to node [left] {\cite{MarquisMuehlherr}} (b5);
\draw[double equal sign distance,arrows=-Implies,thick] (b7) -- (b5);
\draw[double equal sign distance,arrows=Implies-Implies,thick] (b2) to node [right] {(Howlett--M\"uhlherr)} (b4);
\draw[double equal sign distance,arrows=Implies-Implies,thick] (b4) to node [right] {\cite{MarquisMuehlherr}} (b6);
% horizontal
\draw[double equal sign distance,arrows=-Implies,thick] (b5) -- (b6);
\draw[double equal sign distance,arrows=-Implies,thick] (b3) -- (b4);
\draw[double equal sign distance,arrows=Implies-Implies,thick] (b1) to node [auto] {\cref{thm:navigatingverticallayers}} (b2);
\end{tikzpicture}
\end{center}
    \caption{The implications illustrated in this figure follow from (potential) reductions of the isomorphism problem. It illustrates how finding a solution to the isomorphism problem reduces to proving a version of the twist conjecture. Note that some of the implications are not equivalences. 
    }
    \label{fig:reduction}
\end{figure}
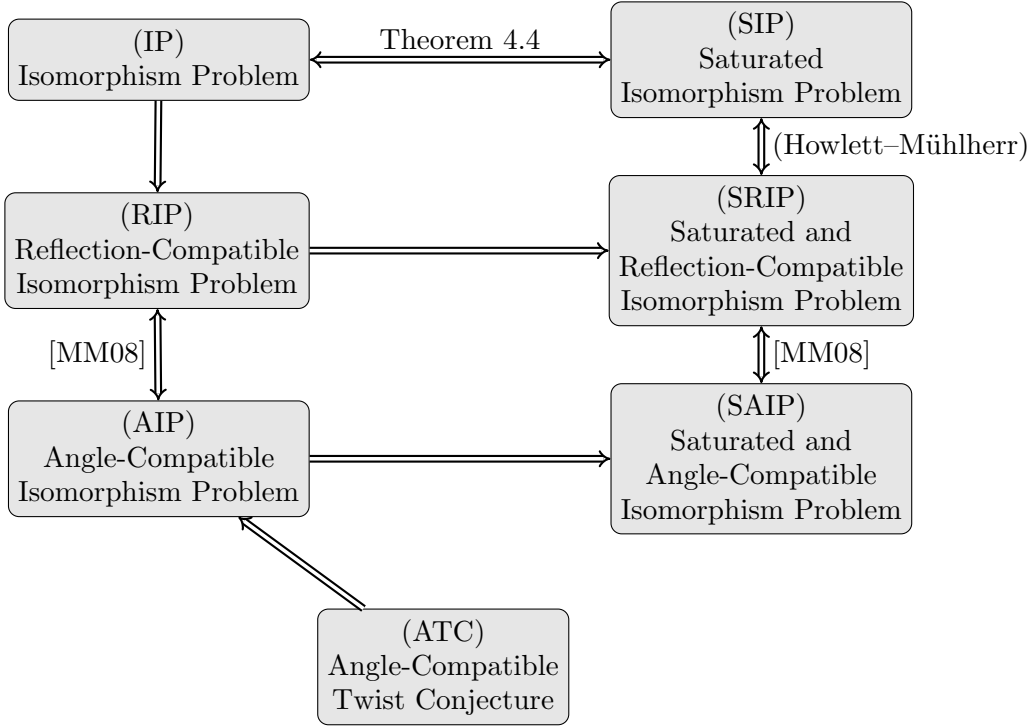

A first possibility to reduce the isomorphism problem to a smaller problem is provided by \cref{thm:navigatingverticallayers}. This theorem allows us to transform any complete Coxeter graph into an saturated graph with isomorphic Coxeter group. 
That is, \cref{qst:isoProblem} reduces to solving the isomorphism problem in the top layer. Hence in \cref{fig:reduction} the two topmost problems (the isomorphism problem and the saturated isomorphism problem) are equivalent. 

\begin{qst}[The saturated isomorphism problem.]
\label{qst:saturatedIsoProblem}
	Given two saturated complete Coxeter graphs, is there an algorithm that decides whether the corresponding vertices of the Coxeter galaxy are in a same connected component? 
\end{qst}

A second reduction is done by passing to the reflection-compatible case, that is considering (SRIP) instead of (RIP) in \cref{fig:reduction}. The saturated version of the isomorphism problem admits  a reduction to the reflection-compatible case by an unpublished theorem of Howlett and M\"uhlherr \cite{HowlettMuehlherr}, also stated as Theorem B.2 in \cite{CapracePrzytyckiTwistRigid}. Let us now state the corresponding definitions and questions. 
By Lemma~3.7 of \cite{BMMN} or Corollary~A.2 of \cite{CapracePrzytyckiTwistRigid}, reflection-compatibility may be defined as follows. 

\begin{dfn}[Reflection-compatible Coxeter systems] \label{def:reflectioncompatible}
	Let $W$ be a Coxeter group and $S$ and $S'$ two Coxeter generating sets for $W$. Then $S$ and $S'$ are \emph{reflection-compatible} if $S^W = (S')^W$, that is, their sets of reflections agree.  
    By slight abuse of notation we will also call the associated graphs and diagrams reflection-compatible. 	
\end{dfn}	

\begin{qst}[(Saturated) reflection-compatible isomorphism problem] 
\label{qst:reflectioncompIsoProblem}
Given two (saturated and) reflection-compatible complete Coxeter graphs, is there an algorithm that decides whether the corresponding vertices of the Coxeter galaxy are in a same connected component? 
\end{qst}

The notion of $J$-deformations (\cref{def:angle deformation}) allows one to reduce \cref{qst:reflectioncompIsoProblem} to the saturated and angle-compatible version, that is \cref{qst:anglecompIsoProblem} below. (See (SRIP) and (SAIP) in the right column of \cref{fig:reduction}). 

\begin{qst}[(Saturated) angle-compatible isomorphism problem] 
\label{qst:anglecompIsoProblem}
Given two (saturated and) angle-compatible complete Coxeter graphs, is there an algorithm that decides whether the corresponding vertices of the Coxeter galaxy are in a same connected component? 
\end{qst}

Putting all these reductions together we have seen that it should suffice to solve the (saturated) angle-compatible twist conjecture in order to prove existence of a solution to the isomorphism problem. 
In doing so one has to rely on the unpublished Theorem~B.2 in \cite{CapracePrzytyckiTwistRigid} of Howlett and M\"uhlherr \cite{HowlettMuehlherr}.  

\begin{prop}
If the (saturated) angle-compatible twist conjecture is true, then the isomorphism problem for Coxeter groups is decidable.
\end{prop}

\begin{proof}
A proof of the angle-compatible twist conjecture implies a solution to the angle-compatible isomorphism problem, which is \cref{qst:anglecompIsoProblem}. 
By Theorem 2 of \cite{MarquisMuehlherr}, a solution to the angle-compatible problem implies that the reflection-compatible one, see \cref{qst:reflectioncompIsoProblem}, is decidable. 
This obviously implies a solution to its saturated version. 
By Howlett and M\"uhlherr's result, Theorem B.2 in \cite{CapracePrzytyckiTwistRigid}, this in turn implies an existing solution to the saturated isomorphism problem stated in \cref{qst:saturatedIsoProblem}. Finally, \cref{thm:navigatingverticallayers} implies that there exists a solution to the (classic) isomorphism problem.  
\end{proof}

%%%%%%%%%%%%%%%%%%%%%%%%%%%%%%%%%%%%%%%%%%%%%%%%%%%%
%%%%%%%%%%%%%%%%%%%%%%%%%%%%%%%%%%%%%%%%%%%%%%%%%%%%
\subsection{Known parts of the galaxy}
\label{subsec:survey}

As it turns out, for many subfamilies of Coxeter groups there is a fairly complete understanding of possible horizontal edges. We finish the paper by collecting major results that allow us to paint a larger picture of the well-known parts of the Coxeter galaxy. Many results have intersections with one another.

%%%%%%%%%%%%%%%%%%%%%%%%%%%%%%%%%
\subsubsection{Notions of rigidity} \label{sec:rigiditynotions}

Following established literature, a Coxeter group $W$ is called \emph{strongly rigid} if any two Coxeter generating sets for $W$ can be bijectively mapped to one another via an inner automorphism of $W$. In particular (cf. \cref{rem:invisibleinfo}), any strongly rigid Coxeter group corresponds to an isolated point in the Coxeter galaxy. 

Starting with geometry, a famous theorem of Charney--Davis~\cite{CharneyDavis} establishes that Coxeter groups with geometric actions on contractible (pseudo-)manifolds are strongly rigid, hence are isolated in the galaxy. 
This was further generalized by Caprace--Przytycki \cite{CapracePrzytyckiBipolar}, who identified the broader geometric condition of bipolarity that still yields strong rigidity. In particular, if $(W,S)$ is bipolar, then $\CCG{W,S} \in \mc{G}$ is isolated. (Since the definition of bipolarity is rather technical --- cf. Theorem~1.2 in \cite{CapracePrzytyckiBipolar} --- we only mention briefly some examples: the groups of Charney--Davis \cite{CharneyDavis}, including cocompact Euclidean or hyperbolic Coxeter groups \cite{DavisReflections}, as well as irreducible, $2$-spherical, nonspherical Coxeter groups \cite{CapraceMuehlherr2-Spherical} are bipolar.) 

The above task of classifying strongly rigid Coxeter groups culminated recently in a paper by Howlett, M\"uhlherr and Nuida. They obtained in \cite{HowlettMuehlherrNuida} a complete characterization of strongly rigid Coxeter groups, and the technical conditions can be read off of any complete Coxeter graph for the given group; see \cite[Theorem~3]{HowlettMuehlherrNuida}. 

In a related paper by Huang and Przytycki \cite{HuangPrzytycki} the notion of $k$-rigid Coxeter systems is introduced. 
Twist rigid systems are isolated points in the top layer of their connected components in case the twist conjecture holds. So in the top layer, twist rigidity is equivalent to rigidity given the twist conjecture. Below the top layer the two concepts differ, as previous examples show. 

Independently of the above mentioned works, early results of Bahls \cite{BahlsThesis} and Mihalik \cite{MihalikEvenIsoThm} also led to a certain rigidity notion for even Coxeter groups: if $W$ is even, then there is a unique even complete Coxeter graph in $\mc{G}$ corresponding to $W$. Moreover, Mihalik describes --- via twists and pseudo-transpositions --- how to check whether a Coxeter group is even and, if so, how to construct its even complete Coxeter graph. Thus, despite even systems not necessarily being isolated in the galaxy, the machinery of Bahls and Mihalik was an early effective solution to both the isomorphism problem (\cref{qst:isoProblem}) and the Reachability \cref{qst:reachability} in this class.

%%%%%%%%%%%%%%%%%%%%%%%%%%%%%%%%%%%%%%%
\subsubsection{Chordal Coxeter groups}

Ratcliffe and Tschantz \cite{RatcliffeTschantz} study a particular class of Coxeter systems called chordal. Within this class of groups the isomorphism problem is solved, the angle-compatible twist conjecture~\ref{conj:ac-twistConjecture} is proven and the Reachability \cref{qst:reachability} is positively answered. 
A Coxeter system is \emph{chordal} if its defining graph (as defined in \cref{def:CoxDiag}) is a chordal graph. 

Theorem~9.2 of \cite{RatcliffeTschantz} states that, given $(W,S)$ and $(W,S')$ two saturated chordal Coxeter systems, then $W$ and $W'$ are isomorphic if and only if the complete Coxeter graph $\CCG{W,S}$ can be deformed into a labeled graph isomorphic to $\CCG{W',S'}$ by a finite sequence of twists and cross-eyed twists. This result also relies on a theorem by Howlett and M\"uhlherr \cite{HowlettMuehlherr}; see Theorem~9.1 in \cite{RatcliffeTschantz}. 

Hence \cref{qst:isoProblem} is solvable within the class of chordal Coxeter groups. Even more is true: the Reachability \cref{qst:reachability} has a positive answer in the case of chordal graphs with a short list of moves: blow-ups and -downs, twists, and cross-eyed twists.

%%%%%%%%%%%%%%%%%%%%%%%%%%%%%%%%%%%%%%%
\subsubsection{Final remarks}

So far we have been able to completely determine how the Coxeter galaxy looks like up to its third layer, resorting essentially to three power tools: matchings of maximal spherical parabolics, finite quotients and profinite rigidity, and geometric actions. Up until $\mc{G}_{\leq 3}$ there are barely any horizontal moves, i.e., edges connecting Coxeter systems in the same layer. There is essentially a single type of horizontal move, namely the edges that arise from blow-ups or -downs. Though as seen in Section~\ref{subsec:twists} (after \cite{BMMN}) the galaxy becomes considerably more complicated from layer four onwards since twists come into play, giving new types of horizontal edges. Nevertheless, a considerable part of the galaxy --- in arbitrary layers --- consists of isolated points due to strong rigidity.

%%%%%%%%%%%%%%%%%%%%%%%%%%%%%%%%%%%%%%%%%%%%%%%%%%%%
%%%%%%%%%%%%%%%%%%%%%%%%%%%%%%%%%%%%%%%%%%%%%%%%%%%%
%%%%%%%%%%%%%%%%%%%%%%%%%%%%%%%%%%%%%%%%%%%%%%%%%%%%
%%%%%%%%%%%%%%%%%%%%%%%%%%%%%%%%%%%%%%%%%%%%%%%%%%%%

\section*{Acknowledgments}
We thank Piotr~Przytycki for pointing out some references. We are indebted to Olga Varghese and the anonymous referee for invaluable remarks, particularly towards correcting and improving the proofs of Theorems~\ref{thm:ProfRigHypRank3} and \ref{thm:analogueBCR}. 
This work was partially supported by the DFG priority program `\emph{Geometry at infinity}' SPP~2026 (Project~62) and the DFG research training group `\emph{MathCoRe}' RTG~2297, 314838170. 

\printbibliography

@UNPUBLISHED{HowlettMuehlherr,
    AUTHOR = {Howlett, R. B. and M{\"{u}}hlherr, B.},
     TITLE = {Isomorphisms of {C}oxeter groups which do not preserve reflections},
	PAGES = {18},
	YEAR = {2004},
	NOTE = {Unpublished preprint, 18 pages},
}

@Article{HowlettMuehlherrNuida,
 Author = {Howlett, R. B. and M{\"u}hlherr, B. and Nuida, K.},
 Title = {Intrinsic reflections and strongly rigid {C}oxeter groups},
 FJournal = {Proceedings of the London Mathematical Society. Third Series},
 Journal = {Proc. Lond. Math. Soc. (3)},
 Volume = {116},
 Number = {3},
 Pages = {534--574},
 Year = {2018},
 DOI = {10.1112/plms.12090},
 }

@article{HuangPrzytycki,
    AUTHOR = {Huang, J. and Przytycki, P.},
     TITLE = {A step towards twist conjecture},
   JOURNAL = {Doc. Math.},
  FJOURNAL = {Documenta Mathematica},
    VOLUME = {23},
      YEAR = {2018},
     PAGES = {2081--2100},
     DOI = {10.25537/dm.2018v23.2081-2100},
}

@article{DavisReflections,
    AUTHOR = {Davis, M. W.},
     TITLE = {Groups generated by reflections and aspherical manifolds not covered by {E}uclidean space},
   JOURNAL = {Ann. of Math. (2)},
  FJOURNAL = {Annals of Mathematics. Second Series},
    VOLUME = {117},
      YEAR = {1983},
    NUMBER = {2},
     PAGES = {293--324},
       DOI = {10.2307/2007079},
}

@article{CapracePrzytyckiBipolar,
    AUTHOR = {Caprace, P.-E. and Przytycki, P.},
     TITLE = {Bipolar {C}oxeter groups},
   JOURNAL = {J. Algebra},
  FJOURNAL = {Journal of Algebra},
    VOLUME = {338},
      YEAR = {2011},
     PAGES = {35--55},
    DOI = {10.1016/j.jalgebra.2011.05.007},
}

@Article{MarquisMuehlherr,
 Author = {Marquis, T. and M{\"u}hlherr, B.},
 Title = {Angle-deformations in {C}oxeter groups},
 FJournal = {Algebraic \& Geometric Topology},
 Journal = {Algebr. Geom. Topol.},
 Volume = {8},
 Number = {4},
 Pages = {2175--2208},
 Year = {2008},
 DOI = {10.2140/agt.2008.8.2175},
}

@Article{RatcliffeTschantz,
 Author = {Ratcliffe, J. G. and Tschantz, S. T.},
 Title = {Chordal {C}oxeter groups.},
 FJournal = {Geometriae Dedicata},
 Journal = {Geom. Dedicata},
 Volume = {136},
 Pages = {57--77},
 Year = {2008},
 DOI = {10.1007/s10711-008-9274-9},
}

@InCollection{NuidaSurvey,
 Author = {Nuida, K.},
 Title = {On the isomorphism problem for {Coxeter} groups and related topics.},
 BookTitle = {Groups of exceptional type, Coxeter groups and related geometries. Invited articles based on the presentations at the international conference on ``Groups and geometries'', Bangalore, India, December 10--21, 2012.},
 Pages = {217--238},
 Year = {2014},
  DOI = {10.1007/978-81-322-1814-2_12},
   Publisher = {New Delhi: Springer},
  }

@article{MihalikTschantzVisual,
    AUTHOR = {Mihalik, M. and Tschantz, S.},
     TITLE = {Visual decompositions of {C}oxeter groups},
   JOURNAL = {Groups Geom. Dyn.},
  FJOURNAL = {Groups, Geometry, and Dynamics},
    VOLUME = {3},
      YEAR = {2009},
    NUMBER = {1},
     PAGES = {173--198},
            DOI = {10.4171/GGD/53},
}

@article{BernhardExample,
    AUTHOR = {M{\"{u}}hlherr, B.},
     TITLE = {On isomorphisms between {C}oxeter groups},
   JOURNAL = {Des. Codes Cryptogr.},
  FJOURNAL = {Designs, Codes and Cryptography. An International Journal},
    VOLUME = {21},
      YEAR = {2000},
    NUMBER = {1-3},
     PAGES = {189},
       DOI = {10.1023/A:1008347930052},
     NOTE = {Special issue dedicated to Dr. Jaap Seidel on the occasion of his 80th birthday (Oisterwijk, 1999)},
}

@article{Deodhar,
    AUTHOR = {Deodhar, V. V.},
     TITLE = {On the root system of a {C}oxeter group},
   JOURNAL = {Comm. Algebra},
  FJOURNAL = {Communications in Algebra},
    VOLUME = {10},
      YEAR = {1982},
    NUMBER = {6},
     PAGES = {611--630},
       DOI = {10.1080/00927878208822738},
}

@article{MuehlherrWeidmann,
    AUTHOR = {M{\"{u}}hlherr, B. and Weidmann, R.},
     TITLE = {Rigidity of skew-angled {C}oxeter groups},
   JOURNAL = {Adv. Geom.},
  FJOURNAL = {Advances in Geometry},
    VOLUME = {2},
      YEAR = {2002},
    NUMBER = {4},
     PAGES = {391--415},
       DOI = {10.1515/advg.2002.018},
}

@article{PFPHyperbolic,
title = {On index 2 subgroups of hyperbolic symmetry groups},
author = {de las Pe{\~n}as, Ma. L. A. N. and Felix, R. P. and Provido, E. D. B.},
pages = {443--448},
volume = {222},
number = {9},
journal = {Z. {K}ristallogr.},
fjournal = {Zeitschrift f{\"u}r {K}ristallographie},
doi = {doi:10.1524/zkri.2007.222.9.443},
year = {2007},
}

@article{ConderDobcsanyi,
    AUTHOR = {Conder, M. and Dobcs{\'a}nyi, P.},
     TITLE = {Applications and adaptations of the low index subgroups procedure},
   JOURNAL = {Math. Comp.},
  FJOURNAL = {Mathematics of Computation},
    VOLUME = {74},
      YEAR = {2005},
    NUMBER = {249},
     PAGES = {485--497},
       DOI = {10.1090/S0025-5718-04-01647-3},
}

@article{Novikov,
    AUTHOR = {Novikov, P. S.},
     TITLE = {Ob algoritmi\v{c}esko\u{\i} nerazre\v{s}imosti problemy to\v{z}destva slov v teorii grupp},
      JOURNAL = {Trudy Mat. Inst. Steklov.},
	FJOURNAL = {Trudy Matematicheskogo Instituta imeni V. A. Steklova},
	VOLUME =  {44},
 PUBLISHER = {Izdat. Akad. Nauk SSSR, Moscow},
      YEAR = {1955},
     PAGES = {3--143},
	URL = {https://www.mathnet.ru/php/archive.phtml?wshow=paper&jrnid=tm&paperid=1180},
}

@book{BourbakiLie4-6,
    AUTHOR = {Bourbaki, N.},
     TITLE = {Lie groups and {L}ie algebras. {C}hapters 4--6},
    SERIES = {Elements of Mathematics (Berlin)},
 PUBLISHER = {Springer-Verlag, Berlin},
      YEAR = {2002},
     PAGES = {xii+300},
	DOI = {10.1007/978-3-540-89394-3},
}

@PHDTHESIS{BahlsThesis,
    AUTHOR = {Bahls, C. P.},
     TITLE = {Even rigidity in {C}oxeter groups},
   school = {Vanderbilt University, USA},
    PUBLISHER = {ProQuest LLC, Ann Arbor, MI},
   year = {2002},
      ISBN = {978-0493-61636-0},
           PAGES = {66},
   }

@UNPUBLISHED{BBM,
	AUTHOR = {Bleak, C. and Brin, M. G. and Moore, J. T.},
	TITLE = {Complexity among the finitely generated subgroups of {T}hompson's group},
	PAGES = {54},
	YEAR = {2017},
	NOTE = {Preprint, arXiv:1711.10998},
	URL = {https://arxiv.org/pdf/1711.10998},
}

@article{BridsonConderReid,
    AUTHOR = {Bridson, M. R. and Conder, M. D. E. and Reid, A. W.},
     TITLE = {Determining {F}uchsian groups by their finite quotients},
   JOURNAL = {Israel J. Math.},
  FJOURNAL = {Israel Journal of Mathematics},
    VOLUME = {214},
      YEAR = {2016},
    NUMBER = {1},
     PAGES = {1--41},
       DOI = {10.1007/s11856-016-1341-6},
}

@article{BMMN,
    AUTHOR = {Brady, N. and McCammond, J. P. and M\"{u}hlherr, B. and Neumann, W. D.},
     TITLE = {Rigidity of {C}oxeter groups and {A}rtin groups},
   JOURNAL = {Geom. Dedicata},
  FJOURNAL = {Geometriae Dedicata},
    VOLUME = {94},
      YEAR = {2002},
     PAGES = {91--109},
       DOI = {10.1023/A:1020948811381},
}

@article{BrieskornSaito,
    AUTHOR = {Brieskorn, E. and Saito, K.},
     TITLE = {Artin-{G}ruppen und {C}oxeter-{G}ruppen},
   JOURNAL = {Invent. Math.},
  FJOURNAL = {Inventiones Mathematicae},
    VOLUME = {17},
      YEAR = {1972},
     PAGES = {245--271},
	DOI = {10.1007/BF01406235},
}

@article{CapracePrzytyckiTwistRigid,
    AUTHOR = {Caprace, P.-E. and Przytycki, P.},
     TITLE = {Twist-rigid {C}oxeter groups},
   JOURNAL = {Geom. Topol.},
  FJOURNAL = {Geometry \& Topology},
    VOLUME = {14},
      YEAR = {2010},
    NUMBER = {4},
     PAGES = {2243--2275},
       DOI = {10.2140/gt.2010.14.2243},
}

@article{CapraceMuehlherr2-Spherical,
    AUTHOR = {Caprace, P.-E. and M{\"{u}}hlherr, B.},
     TITLE = {Reflection rigidity of {$2$}-spherical {C}oxeter groups},
   JOURNAL = {Proc. Lond. Math. Soc. (3)},
  FJOURNAL = {Proceedings of the London Mathematical Society. Third Series},
    VOLUME = {94},
      YEAR = {2007},
    NUMBER = {2},
     PAGES = {520--542},
       DOI = {10.1112/plms/pdl015},
}

@article{CharneyDavis,
    AUTHOR = {Charney, R. and Davis, M.},
     TITLE = {When is a {C}oxeter system determined by its {C}oxeter group?},
   JOURNAL = {J. London Math. Soc. (2)},
  FJOURNAL = {Journal of the London Mathematical Society. Second Series},
    VOLUME = {61},
      YEAR = {2000},
    NUMBER = {2},
     PAGES = {441--461},
       DOI = {10.1112/S0024610799008583},
}

@book{DavisBook,
 Author = {Davis, M. W.},
 Title = {The geometry and topology of {C}oxeter groups},
 FSeries = {London Mathematical Society Monographs Series},
 Series = {Lond. Math. Soc. Monogr. Ser.},
 Volume = {32},
 ISBN = {978-0-691-13138-2},
 Year = {2008},
 Publisher = {Princeton, NJ: Princeton University Press},
}

@article{DyerConjSep,
    AUTHOR = {Dyer, J. L.},
     TITLE = {Separating conjugates in free-by-finite groups},
   JOURNAL = {J. London Math. Soc. (2)},
  FJOURNAL = {Journal of the London Mathematical Society. Second Series},
    VOLUME = {20},
      YEAR = {1979},
    NUMBER = {2},
     PAGES = {215--221},
       DOI = {10.1112/jlms/s2-20.2.215},
}

@article{GrunewaldZalesskiiGenus,
    AUTHOR = {Grunewald, F. and Zalesskii, P.},
     TITLE = {Genus for groups},
   JOURNAL = {J. Algebra},
  FJOURNAL = {Journal of Algebra},
    VOLUME = {326},
      YEAR = {2011},
     PAGES = {130--168},
       DOI = {10.1016/j.jalgebra.2010.05.018},
}

@article{HughesFuchsian,
    AUTHOR = {Hughes, S.},
     TITLE = {Cohomology of {F}uchsian groups and non-{E}uclidean crystallographic groups},
   JOURNAL = {Manuscripta Math.},
  FJOURNAL = {Manuscripta Mathematica},
    VOLUME = {170},
      YEAR = {2023},
    NUMBER = {3-4},
     PAGES = {659--676},
       DOI = {10.1007/s00229-022-01369-z},
}

@book{KatokFuchsian,
    AUTHOR = {Katok, S.},
     TITLE = {Fuchsian groups},
    SERIES = {Chicago Lectures in Mathematics},
 PUBLISHER = {University of Chicago Press, Chicago, IL},
      YEAR = {1992},
     PAGES = {x+175},
      ISBN = {978-0-226-42583-5},
   MRCLASS = {20H10 (30F35)},
}

@article{LiebeckShalevFuchsian,
    AUTHOR = {Liebeck, M. W. and Shalev, A.},
     TITLE = {Fuchsian groups, finite simple groups and representation varieties},
   JOURNAL = {Invent. Math.},
  FJOURNAL = {Inventiones Mathematicae},
    VOLUME = {159},
      YEAR = {2005},
    NUMBER = {2},
     PAGES = {317--367},
       DOI = {10.1007/s00222-004-0390-3},
}

@article{MihalikEvenIsoThm,
    AUTHOR = {Mihalik, M.},
     TITLE = {The even isomorphism theorem for {C}oxeter groups},
   JOURNAL = {Trans. Amer. Math. Soc.},
  FJOURNAL = {Transactions of the American Mathematical Society},
    VOLUME = {359},
      YEAR = {2007},
    NUMBER = {9},
     PAGES = {4297--4324},
       DOI = {10.1090/S0002-9947-07-04133-5},
}

@article{MihalikRatcliffeTschantz,
    AUTHOR = {Mihalik, M. L. and Ratcliffe, J. G. and Tschantz, S. T.},
     TITLE = {Matching theorems for systems of a finitely generated {C}oxeter group},
   JOURNAL = {Algebr. Geom. Topol.},
  FJOURNAL = {Algebraic \& Geometric Topology},
    VOLUME = {7},
      YEAR = {2007},
     PAGES = {919--956},
       DOI = {10.2140/agt.2007.7.919},
}

@article{MihalikRatcliffe-Ranks,
    AUTHOR = {Mihalik, M. L. and Ratcliffe, J. G.},
     TITLE = {On the rank of a {C}oxeter group},
   JOURNAL = {J. Group Theory},
  FJOURNAL = {Journal of Group Theory},
    VOLUME = {12},
      YEAR = {2009},
    NUMBER = {3},
     PAGES = {449--464},
       DOI = {10.1515/JGT.2008.089},
}

@incollection{BernhardSurvey,
    AUTHOR = {M{\"{u}}hlherr, B.},
     TITLE = {The isomorphism problem for {C}oxeter groups},
 BOOKTITLE = {The {C}oxeter legacy},
     PAGES = {1--15},
 PUBLISHER = {Amer. Math. Soc., Providence, RI},
      YEAR = {2006},
    ISBN = {0-8218-3722-2},
}

@article{NuidaIndecomp,
	AUTHOR = {Nuida, K.},
	TITLE = {On the direct indecomposability of infinite irreducible {C}oxeter groups and the isomorphism problem for {C}oxeter groups},
	JOURNAL = {Comm. Algebra},
	FJOURNAL = {Communications in Algebra},
	VOLUME = {34},
	YEAR = {2006},
	NUMBER = {7},
	PAGES = {2559--2595},
	DOI = {10.1080/00927870600651281},
}

@article{Ratcliffe,
    AUTHOR = {Ratcliffe, J. G.},
     TITLE = {Finiteness conditions for groups},
   JOURNAL = {J. Pure Appl. Algebra},
  FJOURNAL = {Journal of Pure and Applied Algebra},
    VOLUME = {27},
      YEAR = {1983},
    NUMBER = {2},
     PAGES = {173--185},
       DOI = {10.1016/0022-4049(83)90013-0},
}

@book{RibesZalesskii,
    AUTHOR = {Ribes, L. and Zalesskii, P.},
     TITLE = {Profinite groups},
    SERIES = {Ergebnisse der Mathematik und ihrer Grenzgebiete. 3. Folge},
    VOLUME = {40},
   EDITION = {Second},
 PUBLISHER = {Springer-Verlag, Berlin},
      YEAR = {2010},
     PAGES = {xvi+464},
       DOI = {10.1007/978-3-642-01642-4},
}

@inproceedings{ReidICM,
    AUTHOR = {Reid, A. W.},
     TITLE = {Profinite rigidity},
 BOOKTITLE = {Proceedings of the {I}nternational {C}ongress of {M}athematicians, {R}io de {J}aneiro},
     PAGES = {1193--1216},
 PUBLISHER = {World Sci. Publ., Hackensack, NJ},
      YEAR = {2018},
      Volume = {II},
      DOI = {10.1142/11060},
}

@misc{Roberts,
  author = {Roberts, S.},
  title = {{Donald Coxeter: The man who saved geometry}},
  howpublished = "Toronto Live, \url{https://www.math.toronto.edu/mpugh/Coxeter.pdf}",
  year = {2003}, 
  note = "[Online; accessed 05-Oct-2022]"
}

@book{SerreTrees,
    AUTHOR = {Serre, J.-P.},
     TITLE = {Trees},
      NOTE = {Translated from the French by John Stillwell},
 PUBLISHER = {Springer-Verlag, Berlin-New York},
      YEAR = {1980},
     PAGES = {ix+142},
     DOI = {10.1007/978-3-642-61856-7},
}

@article{WeigelTwistTriangle,
    AUTHOR = {Weigel, C. J.},
     TITLE = {The twist conjecture for {C}oxeter groups without small triangle subgroups},
   JOURNAL = {Innov. Incidence Geom.},
  FJOURNAL = {Innovations in Incidence Geometry. Algebraic, Topological and
              Combinatorial},
    VOLUME = {12},
      YEAR = {2011},
     PAGES = {111 - 140},
    DOI = {10.2140/iig.2011.12.111},
}

\end{document}